\documentclass{amsart}
\usepackage{geometry}
\usepackage[english]{babel}
\usepackage[latin1]{inputenc}
\usepackage[sc]{mathpazo}
\usepackage[colorlinks]{hyperref}
\usepackage{amssymb,amsmath,amsthm,amsfonts, mathrsfs, mathtools}

\usepackage{enumitem}

\def\R {\mathbb{R}}
\def\N {\mathbb{N}}
\def\C {\mathcal{C}}
\def\eps{\varepsilon}

\renewcommand{\div}{\mathrm{div}}

\renewcommand{\Re}{\mathrm{Re}}

\newcommand{\id}{\mathrm{Id}}
\newcommand{\elle}{\mathcal{L}}

\DeclareMathOperator{\dist}{dist}

\newtheorem{proposition}{Proposition}[section]
\newtheorem{theorem}[proposition]{Theorem}
\newtheorem*{theorem*}{Theorem}

\newtheorem{lemma}[proposition]{Lemma}

\newtheorem*{claim}{Claim}
\theoremstyle{definition}
\newtheorem{definition}[proposition]{Definition}
\newtheorem{remark}[proposition]{Remark}
\numberwithin{equation}{section}
\title[Prey-predators, Part I]{Predators-prey models with competition \\ Part I: existence, bifurcation and qualitative properties}
\author{Henri Berestycki}
\email{hb@ehess.fr}
\address{\'{E}cole des Hautes \'{E}tudes en Sciences Sociales, PSL Research University Paris, Centre d'analyse et de math\'{e}matique sociales (CAMS), CNRS, 54 bouvelard Raspail, 75006, Paris}

\author{Alessandro Zilio}
\email{azilio@math.univ-paris-diderot.fr}
\address{Universit\'{e} Paris Diderot - Paris 7, Laboratoire J.-L.\ Lions (CNRS UMR 7598), Paris, France, 8 place Aur\'elie Nemours, 75205, Paris CEDEX 13}


\begin{document}

\begin{abstract}
We study a mathematical model of environments populated by both preys and predators, with the possibility for predators to actively compete for the territory. For this model we study existence and uniqueness of solutions, and their asymptotic properties in time, showing that the solutions have different behavior depending on the choice of the parameters. We also construct heterogeneous stationary solutions and study the limits of strong competition and abundant resources. We then use these information to study some properties such as the existence of solutions that maximize the total population of predators. We prove that in some regimes the optimal solution for the size of the total population contains two or more groups of competing predators.
\end{abstract}

\maketitle

\section{Introduction}

Systems of reaction-diffusion equations are ubiquitous in mathematical biology, as they serve as a basic framework for modeling a diversity of biological and ecological mechanisms. In particular,  the study of population dynamics often involves such systems. In a recent paper \cite{BZ_ecology}, we have introduced a new reaction-diffusion system to describe the emergence of territoriality in predatory animals. As a matter of fact, this question, to a large extent, remains puzzling. Specifically, our aim was to understand whether selfish and non organized behaviors of predators are sufficient mechanisms to explain the emergence of territoriality. 

In this model, we only consider basic mechanisms that characterize an environment inhabited by predators and preys. Given a region $\Omega \subset \R^n$ with $n\leq 2$ in practice (the restriction on the dimension is purely for modeling reasons) occupied by $N+1$ groups of animals. The density of the first one, which we denote as $u$, of preys while the remaining $N$ densities, denoted by $w_1, \dots, w_N$, of predators. Each of the densities evolves in time following a typical law of Lotka-Volterra type. The model proposed in \cite{BZ_ecology} is synthesized in the form of the system
%\begin{subequations}\label{eqn model}
\begin{equation}\label{eqn model}
	\begin{cases}
		w_{i,t} - d_i \Delta w_i  = \left(- \omega_i + k_i u -\mu_i w_i- \beta \sum_{j \neq i} a_{ij} w_j\right) w_i \\
		u_{t} - D \Delta u = \left(\lambda - \mu u - \sum_{i=1}^N k_i w_i \right)u
	\end{cases}
\end{equation}
for $(x,t) \in \Omega \times (0,+\infty)$, completed by homogeneous Neumann boundary and smooth initial conditions. The parameters of the model are easily explained: some terms in the equations model internal mechanism in the populations
\begin{itemize}
	\item $D, d_1, \dots, d_N$ are the diffusion coefficients of the different populations, and thus are always considered positive in the following;
	\item $\lambda > 0$ is the effective reproduction rate of the preys, and $\mu \geq 0$ stands for the possible saturability of the environment due to an excess of preys;
	\item $\omega_i$ is the mortality coefficient of the group $i$ of  predators that takes into account the starvation caused by the absence of the prey $u$, and $\mu_i$ takes into account possible saturation phenomena in the predator populations (for instance, an internal --low level-- competition between member of the same density).
\end{itemize}
The other terms are, on the other hand, responsible for the interaction between different densities
\begin{itemize}
	\item $k_1, \dots, k_N$ govern the predation rates. That is, $k_i$ is the success of predator $w_i$ in catching prey $u$ as a factor of the probability of an encounter;
	\item the elements $a_{ij} > 0$ represent how the presence of the density $w_j$ affects the density $w_i$, and the particular choice of the sign suggest that we only consider competing interactions. The parameter $\beta \geq 0$ on the other hand expresses the strength of the interaction: the higher the value of $\beta$, the more aggressive is the behavior of predators between different groups. 
\end{itemize}

Similar models have already been introduced in the ecological and mathematical literature, starting from the seminal paper by Volterra \cite{Volterra} on predator-preys interactions, to the more recent contribution \cite{DancerDu} that started the study of strongly interaction systems of elliptic equations (in that case, modeling populations of competing predators without preys),  the study on the evolution of dispersal by means of systems of many interacting predators \cite{Dockery}, and the papers on the qualitative properties of the solutions to such systems \cite{ContiTerraciniVerzini_AdvMat_2005, CaffKarLin, DaWaZa_Dynamics}. The novelty in our model, that complicates the analysis but allows for more profound results, is the inclusion in the system of the equation for the preys and of the competition between the predators. A more in depth comparison with the results in the scientific literature can be found in \cite{BZ_ecology}, to which we refer the interested reader for more detailed biological considerations.

The main results in this paper regarding model \eqref{eqn model} are summarized as follows. In the following sections we will give more general statements. For sufficiently smooth and positive initial data, the system \eqref{eqn model} admits a unique, bounded and smooth solution, defined for all $t \geq 0$ (see Lemma \ref{lem existence and bounds}). The competition is a driving force in the heterogeneity of the set of solutions, indeed the set of stationary solutions of the system collapse to the set of constant solutions if $\beta$ is small (see Proposition \ref{prp conway}), and is very rich for $\beta$ large (see for instance Theorems \ref{thm bif}, \ref{thm bif an} and \ref{thm bif one}), but under some assumptions the asymptotic behavior $\beta \to +\infty$ can be described accurately (see Lemma \ref{lem existence and bounds} and Proposition \ref{prp conway}). The solutions of the stationary system are regular, independently of the strength of the competition, and they converge to segregated configurations when $\beta \to +\infty$ (see Proposition \ref{prp asymptotic k}) in which territories of different competing groups do not overlap. This allows us to define also limit solutions to \eqref{eqn model} in the case $\beta = +\infty$. Aggressiveness ($\beta \gg 1$) may help to resist an invasion by a foreign group (see Propositions \ref{prop dock new} and \ref{prp dock full}).

On the other hand, the strong competition limits the number of different groups that can survive in a given domain. Indeed we have
\begin{theorem*}
For a given smooth domain $\Omega \subset \R^n$, there exist $\bar N \in \N$ and $\bar \beta > 0$ such that if $\beta > \bar \beta$ and $\mathbf{w}_\beta = (w_{1,\beta}, \dots, w_{N,\beta}, u_{\beta})$ is a solution to \eqref{eqn model} then
\begin{itemize}
	\item either at most $\bar N$ components of $(w_{1,\beta}, \dots, w_{N,\beta})$ are strictly positive and the others are zero;
	\item or the solution is such that
\[
	 \|(w_{1,\beta},\dots,w_{N,\beta})\|_{\C^{0,\alpha}(\Omega)} + \|u_\beta- \lambda/\mu\|_{\C^{2,\alpha}(\Omega)} = o_\beta(1)
\]
for every $\alpha \in (0,1)$.
\end{itemize}
Furthermore, the threshold value $\bar N$ has the following upper-bound:
\[
	\bar N \lesssim \frac{|\Omega|}{4 \pi} \max_{i = 1, \dots, k} \frac{\lambda k_i - \mu \omega_i}{d_i \mu} 
\]
if $n =2$, and similar estimates hold in any dimension.
\end{theorem*}

In the theorem $A \lesssim B$ stands for $A \leq B + o(B)$ where $o(B)/B \to 0$ as $B \to +\infty$. This theorem states that, when $\beta$ is sufficiently large, solutions to \eqref{eqn model} are either close to constant (and small) solutions, or they have at most $\bar N$ (+1) non trivial components. This result has important repercussions on the biological interpretations of the model, as it imposes an upper-bound on the total number of hostile groups of predators that can survive in a given environment. Moreover, the upper-bound itself has important ecological consequences: We have explored them in \cite{BZ_ecology}. 

Finally, under some assumptions on the coefficients, there exist a number of densities of predators $N \in \N_0$ and a solution $(w_1, \dots, w_N, u)$ of \eqref{eqn model} that maximize the total population of predators. We show furthermore that, in many cases, this optimal configuration is given by two or more densities of predators that have very aggressive behavior between each other, rather than by a simple homogeneous group that displays no aggressiveness between its components.
 
\begin{theorem*}[see Theorems \ref{thm maxim} and \ref{2 is better}]
For any given smooth domain $\Omega$, there exist a number $\bar N \in \N$ and a solution $(w_1, \dots, w_N, u)$ of \eqref{eqn model}  at most $\bar N +1$ non trivial components (possibly with $\beta = +\infty$) that maximizes the functional
\[
	P(w_1, \dots, w_N, u) = \int_{\Omega} \sum_{i=1}^N w_i
\]
among the set of all non negative solutions of \eqref{eqn model}. Moreover, if $\Omega$ is a rectangular domain (in any dimension) and $\mu$ is sufficiently small, the maximum is attained by a solution with two or more densities of predators, that is, $\bar N \geq 2$ and $\beta > 0$.
\end{theorem*}

An interesting open problem remains for the second conclusion of the theorem, concerning the case when $\bar N \geq 2$. Indeed, we believe that the result holds true for rather general domains, and numerical simulations sustain our claim, but the proof is so far elusive. 

A consequence of the theorem is that competition between predators can be beneficial not only for the preys, but  for the total population of predators as well. 

\subsection*{Structure of the paper} The paper is organized as follows: in Section \ref{sec long time} we consider some basic properties of the system, such as existence and regularity of solutions, together with some asymptotic properties of the system, focusing on the stability properties of specific solutions. In Section \ref{sec bif}, thanks to a bifurcation analysis, we show that the set of stationary solutions is very rich. In Section \ref{sec unif est}, we present some uniform estimates that we later use to give a precise description of the solutions for large competition. In Section \ref{sec dim one}, we show a more precise description of the bifurcation diagram in dimension one. In Section \ref{sec opt rep}, we investigate some properties of the system with a large number of components, and finally, in Section \ref{sec emergence}, we show configurations in which the maximizers of the integral of the densities $w_i$ are spatially heterogeneous.\hfill$\blacksquare$

The interested reader will find a companion paper \cite{BZ_ecology} to this one, where we investigate the biological and ecological interpretations of the mathematical results herein contained. The present paper is the first of a series in which we investigate properties of the system \eqref{eqn model}. In a second part \cite{BerestyckiZilio_NN} we give deeper a priori estimates of the solution to the elliptic counterpart of system \eqref{eqn model}. We exploit these properties in \cite{BerestyckiZilio_RR} to give a more precise description of the set of solutions. Finally, in \cite{BerestyckiZilio_FF} we prove results regarding the parabolic version of the model.

\section{Basic properties of the solutions}\label{sec long time}
In this section we investigate some basic properties of the system. First, we establish existence and uniqueness results for the solutions. Then we analyze the long time behavior of the set of solutions. We also consider stability properties of a special class of solutions, namely those with only one predator and one prey, that is where all the $w_i$ are zero but one. 

We recall that the system reads:
\begin{subequations}\label{eqn model long time}
\begin{equation}\label{eqn model eqn}
	\begin{cases}
		w_{i,t} - d_i \Delta w_i  = \left(- \omega_i + k_i u -\mu_i w_i- \beta \sum_{j \neq i} a_{ij} w_j\right) w_i \\
		u_{t} - D \Delta u = \left(\lambda - \mu u - \sum_{i=1}^N k_i w_i \right)u
	\end{cases}
\end{equation}
in a domain $Q := \Omega \times (0,\infty)$, with $\Omega \Subset \R^n$ open, smooth, bounded and connected. It is completed by boundary and smooth initial conditions
\begin{equation}\label{boundary cond}
	\begin{cases}
		\partial_\nu w_i = \partial_\nu u = 0 &\text{ on $\partial \Omega \times (0,+\infty)$}\\
		w_i(x,0) = w_i^0(x) \geq 0, u(x,0)=u^0(x) \geq 0 &\text{ on $\Omega \times \{0\},$}
	\end{cases}
\end{equation}
\end{subequations}
where $\nu$ denotes the unit outward normal vector field on $\partial \Omega$. We start with the following existence result
\begin{lemma}\label{lem existence and bounds}
Let $(w_1^0, \dots, w_N^0, u^0) \in \C^{0,\alpha}(\Omega)$ be a non-negative initial condition for the system \eqref{eqn model long time}. There exists a unique solution $(w_1, \dots, w_N, u) \in \C^{2,\alpha}_x\C^{1,\alpha/2}_t(Q)$ for all $\alpha \in (0,1)$ which is defined globally for all $t > 0$. Moreover the solution is  bounded in $L^\infty(Q)$ and for any $\eps>0$ there exists $T_\eps>0$ such that
\[
	\begin{split}
		\sup_{(x,t) \in \Omega \times [T_\eps,+\infty)} w_i(x,t) &\leq \frac{\lambda k_i - \mu \omega_i}{\mu \mu_i}  + \eps\\
		\sup_{(x,t) \in \Omega \times [T_\eps,+\infty)} u(x,t) &\leq \frac{\lambda}{\mu} + \eps.
	\end{split}
\]
Consequently, if there exists and index $i \in \{1,\dots,N\}$ such that $\lambda k_i \leq \mu \omega_i$, then
\[
	\lim_{t\to +\infty} \sup_{x \in \Omega} w_{i}(x,t) = 0.
\]
\end{lemma}
In view of the last property, we shall also assume in the following that
\begin{itemize}
	\item[(H)] \label{assumpt} the relation $\lambda k_i > \mu \omega_i$ holds for all $i = 1,\dots,N$.
\end{itemize}

\begin{proof}
The existence of solution for $t \in [0, t_0]$ with $t_0 > 0$ small follows from standard arguments, since the semi-linear terms of the system are locally Lipschitz continuous: in order to extend the existence result for all time $t>0$, it suffices to show an a priori $L^\infty$ uniform bound on the solutions.

First of all, we can observe that each single equation of the system \eqref{eqn model long time} is satisfied by the trivial solution ($u=0$ and $w_i = 0$ for some $i$). Consequently, the comparison principle applied to each equation implies that the solutions, when defined, are strictly positive for positive $t$. Using this information, we focus our attention on the equation satisfied by the density $u$, that is
\begin{equation}\label{eqn u from sys}
	\begin{cases}
		u_{t} - D \Delta u = \left(\lambda - \mu u -  \sum_{i=1}^N k_i w_i\right)u \\
		u(x,0) = u^0(x).
	\end{cases}
\end{equation}
Let $U \in \C^1(\R^+)$ be the solution of the initial value problem
\[
	\begin{cases}
		\dot U = \lambda U - \mu U^2 &\text{for $t >0$}\\
		U(0) = \max \{ \lambda/\mu , \sup_{x\in \Omega} u^0(x) \} > 0.
	\end{cases}
\]
The family of solution $U$ is decreasing in $t >0$ and $U(t) \to \lambda/\mu$ as $t \to +\infty$: as a result, for any $\eps > 0$ there exists $T_\eps \geq 0$ finite such that $U(t) \leq \lambda/\mu + \eps$ for any $t \geq T_\eps$. Clearly, since each $w_i$ is non-negative, $u(x,t) \leq U(t)$ for all $x \in  \Omega$. Therefore, $u(x,t)$ is then bounded uniformly. Taking into account this information, we see that each $w_i$ satisfies 
\begin{equation}\label{eqn wi from sys}
	\begin{cases}
		w_{i,t} - d_i \Delta w_i \leq \left(- \omega_i + k_i U - \mu_i w_i - \beta \sum_{j \neq i} a_{ij} w_j \right) w_i \\	
		w_i(x,0) = w_i^0(x).
	\end{cases}
\end{equation}
Using a similar reasoning as before, we can introduce the auxiliary function $W_{i} \in \C^1(\R^+)$ solution to the initial value problem
\[
	\begin{cases}
		\dot W_i = (-\omega_i + k_i U - \mu_i W_i) W_i &\text{for $t >0$}\\
		W_i(0) =  \max \left\{ \frac{\lambda k_i - \mu \omega_i}{\mu \mu_i}, \sup_{x\in \Omega} w_i^0(x) \right\} > 0.
	\end{cases}
\]
Clearly $W_i$ is uniformly bounded in $t$ and moreover, $W_i(t) \to (\lambda k_i - \mu \omega_i)/ (\mu \mu_i)$ as $t \to +\infty$. Hence, we deduce that for any $\eps > 0$ there exists $T_\eps \geq 0$ finite such that $W_i(t) \leq  (\lambda k_i - \mu \omega_i)/ (\mu \mu_i) + \eps$ for any $t \geq T_\eps$.  Again, since each $w_i$ is non negative and $u \leq U$, we see that $W_i$ is a super-solution for \eqref{eqn wi from sys} and thus $w_i$ is bounded uniformly.

The previous uniform upper bounds are enough to ensure that the solution can be extended for all time $t > 0$ and also yield the asymptotic estimates. 
\end{proof}

Before going further we recall a result in \cite{ConwayHoffSmoller} about the asymptotic behavior in time of solutions to systems of reaction diffusion equations. We let $L$ be the Lipschitz constant of the semi-linear term in \eqref{eqn model} on the invariant region of Lemma \ref{lem existence and bounds}. That is, letting
\[
	F(s_1, \dots, s_N, S) = 
	\left( \begin{array}{c} \left(- \omega_i + k_i S -\mu_i s_i- \beta \sum_{j \neq i} a_{ij} s_j\right) s_i \\
		\left(\lambda - \mu S - \sum_{i=1}^N k_i s_i \right)S
	\end{array}\right)
\]
we define
\[
	L = \sup\left\{ |\nabla F(s_1, \dots, s_N, S)| ; 0 < s_i < \frac{\lambda k_i - \mu \omega_i}{\mu \mu_i}, \; 0 < S <\frac{\lambda}{\mu} \right\}.
\]
We observe that, thanks to the assumptions, $L$ is finite and positive. We also let
\[
	d = \min\{ d_1, \dots, d_N, D\}
\]
and define $\gamma_1$ to be the first non trivial (that is, positive) eigenvalue of the Laplace operator $-\Delta$ in $\Omega$ with homogeneous boundary conditions. Finally, for any solution $(w_1, \dots, w_N, u)$ of \eqref{eqn model}, we let
\[
	\bar w_i(t) = \frac{1}{|\Omega|} \int_\Omega w_i(x,t) dx, \quad \bar u(t) =  \frac{1}{|\Omega|} \int_\Omega u(x,t) dx.
\]
Applying \cite[Theorem 3.1]{ConwayHoffSmoller} to our system \eqref{eqn model} we have the following result on the asymptotic behavior of the solutions for large time.
\begin{proposition}\label{prp conway}
Let 
\[
	\sigma = d \gamma_1 - L.
\]
If $\sigma > 0$, then for any non negative initial condition $(w_1^0, \dots, w_N^0, u^0) \in \C^{0,\alpha}(\Omega)$, the corresponding unique solution of the system \eqref{eqn model long time} converges exponential towards spatially homogeneous solutions, that is, for any $0 < \sigma' < \sigma$ there exists a constant $C > 0$ such that
\begin{gather*}
	\sum_{i=1}^N \| \nabla w_i \|_{L^2(\Omega)}  + \| \nabla u \|_{L^2(\Omega)} \leq C e^{-\sigma' t}\\
	\sum_{i=1}^N  \left\| w_i(\cdot,t) -  \bar w_i(t) \right\|_{L^\infty(\Omega)} + \left\| u(\cdot,t) -  \bar u(t) \right\|_{L^\infty(\Omega)}  \leq C e^{-\sigma' t /n}.
\end{gather*}
Moreover, the vector $(\bar w_1, \dots, \bar w_M, \bar u)$ is a solution of a  the system of ordinary differential equations of the form
\[
	\begin{cases}
		\bar w_{i}' = \left(- \omega_i + k_i \bar u -\mu_i \bar w_i- \beta \sum_{j \neq i} a_{ij} \bar w_j\right) \bar w_i + g_i(t)\\
		\bar u_{t}' = \left(\lambda - \mu \bar u - \sum_{i=1}^N k_i \bar w_i \right)\bar u + g(t)
	\end{cases}
\]
with
\[
		\bar w_i(0) = |\Omega|^{-1} \int_\Omega w_i(x)^0 dx, \quad \bar u(0) =  |\Omega|^{-1} \int_\Omega u(x)^0 dx.
\]
and
\[
	\sum_{i=1}^N |g_i(t)| + |g(t)| \leq C e^{-\sigma' t}.
\]
\end{proposition}
\begin{proof}
The proof is a straightforward application of \cite[Theorem 3.1]{ConwayHoffSmoller}. We only observe that by Lemma \ref{lem existence and bounds} we know that from any positive initial data and any $\eps > 0$ there exists $T_\eps > 0$ such that the corresponding unique solution is contained in the region 
\[
	\left\{ 0 < w_i(x,t) < \frac{\lambda k_i - \mu \omega_i}{\mu \mu_i} + \eps, \; 0 < u(x,t) <\frac{\lambda}{\mu} + \eps, \; \forall x \in \Omega \right\}
\]
for all $t \geq T_\eps$. Now, if $\sigma > 0$, by regularity of $F$ for any $\eps > 0$ sufficiently small
\[
	\sigma' = d \gamma_1 -  \sup\left\{ |\nabla F(s_1, \dots, s_N, S)| ; 0 < s_i < \frac{\lambda k_i - \mu \omega_i}{\mu \mu_i} + \eps, \; 0 < S <\frac{\lambda}{\mu} + \eps \right\} > 0
\]
and we can apply \cite[Theorem 3.1]{ConwayHoffSmoller} to obtain the stated exponential estimates.
\end{proof}

The important consequence of the previous proposition is that the behavior of the solutions, in the regime $\sigma > 0$ is well described by the corresponding system of ordinary differential equations. It also gives us a complete characterization of the set of stationary solutions of \eqref{eqn model long time}, which is then given by the (spatially constant) solutions of $F(w_1, \dots, w_N, u) = 0$. For instance, by studying the stability of the stationary and homogeneous solutions (see Proposition \ref{prop dock new} and Lemma \ref{lem matrix A} below) we will see that in this case, when $\beta > 0$ the only stable stationary solutions are those that have $u > 0$ and only one component of $(w_1, \dots, w_N)$ non trivial (and positive). 

Finally, we observe that the condition $\sigma > 0$ can be violated in three different ways:
\begin{enumerate}[label=(\roman*)]
	\item lowering the diffusion coefficients,
	\item enlarging the domain or
	\item augmenting the Lipschitz constant $L$. 
\end{enumerate}
This last possibility, which is the one that we mainly explore later, can be result for instance form taking $\mu$ small and $\beta$ large enough.

We now start investigating the equilibria of the system, in particular we want to analyze what is the impact of the competition parameter on the possible heterogeneity of the solutions of the system.

We first recall the well known result by Dockery~et~al.\ \cite{Dockery} on a related simpler model
\begin{equation}\label{eqn dock}
	\begin{cases}
		w_{i,t} - d_i \Delta w_i = \left(a(x) -  \sum_{j =1}^N w_j\right) w_i  &\text{in $\Omega \times (0,+\infty)$}\\
		\partial_\nu w_i = 0 &\text{on $\partial\Omega \times (0,+\infty)$}.
	\end{cases}
\end{equation}
This system describes $N$ populations that share the same spatially distributed resource $a$. These populations do not compete actively against each other, but they do suffer from overpopulation, which is model by the logistic term in the equations. Here $a$ is a smooth non constant function such that the principal eigenvalue of each of the elliptic operators
\[
	\begin{cases}
		-d_i \Delta w = a w + \lambda w &\text{in $\Omega$}\\
		\partial_\nu w = 0 &\text{on $\partial \Omega$,}
	\end{cases}
\]
denoted by $\lambda(d_i,a)$, is strictly negative (implying, in particular, the instability of the zero solution). Exploiting the particular symmetric structure of the interaction/competition term, Dockery~et~al.\ were able to show that if $a$ is not constant, the only asymptotically stable equilibrium of the system is the stationary solution that has all the components $w_i$ zero except for the one with the smallest diffusion coefficient $d_i$. Moreover, the same result holds if we introduce a small mutation term in the system, which in terms imply also an evolutionary advantage for small diffusion rates. The classic interpretation of this result is that, since the densities $w_i$ in \eqref{eqn dock} are equivalent if not for the diffusion rates, the density which can concentrate more on favorable zones (maxima of $a$) will benefit more than the others and will end up eliminating them.

In what follows, we shall show that this is not the case for the model we are considering, and in particular we prove that for $\beta$ sufficiently large, all the solutions that have only one nontrivial density of predators are asymptotically stable. 

\begin{remark} In order to justify the link between the model \eqref{eqn dock} and our model \eqref{eqn model long time}, let us consider the limit case of \eqref{eqn model long time} in which the density $u$ has a very fast dynamic with respect to the other components, that is, let us assume that for each $t>0$, the density $u$ reaches instantaneously its non-trivial inviscid equilibrium state,
\[
	\lambda u - \mu u^2 - u \sum_{i=1}^N k_i w_i = 0 \implies u = \frac{1}{\mu}\left(\lambda - \sum_{i =1}^N k_i w_i\right).
\]
Substituting the previous identity in the equations satisfied by $w_i$ we obtain
\[
	w_{i,t}-d_i \Delta w_i = \left( \frac{k_i \lambda}{\mu} - \omega_i  - \frac{k_i}{\mu}w_i - \sum_{j\neq i} \left(\beta a_{ij} + \frac{k_i}{\mu}\right)w_j \right) w_i
\]
In the simplified case $k_i = \mu$, $\omega_i = \omega$ and $\beta = 0$, we obtain finally
\[
	w_{i,t}-d_i \Delta w_i = \left( \lambda - \omega  -  \sum_{j = 1}^N w_j \right) w_i = \left( a -  \sum_{j = 1}^N w_j \right) w_i.
\]
We thus obtain the model of Dockery~et~al.\ \cite{Dockery} with $a = \lambda - \omega$. Notice that we could consider that $\lambda$ and $\omega$ depend on the location in space (certain locations being more favorable than others). More details can be found in \cite{BZ_ecology}.
\end{remark}
We have the following

\begin{proposition}\label{prop dock new}
For a fixed $i \in \{1,\dots,N\}$, let $W$ be the stationary solution of \eqref{eqn model} which has only the $i$-th densities of predator which is nontrivial. Then $\mathbf{v}$ is constant and $\mathbf{v} = (0, \dots, \tilde w_i, \dots, 0, \tilde u)$ with
\[
	\tilde w_i = \frac{\lambda k_i - \mu \omega_i}{k_i^2}, \quad \tilde u = \frac{\omega_i}{k_i}.
\]
There exists $\bar \beta \geq 0$ such that if $\beta \geq \bar \beta$ then $W$ is asymptotically stable with respect to perturbations in $\C^{2,\alpha}(\overline{\Omega})$. More explicitly, this holds whenever $\beta$ satisfies the system of inequalities
\[
	\beta  \geq \frac{k_j}{a_{ji} \tilde w_i} \left( \frac{\omega_i}{k_i} - \frac{\omega_j}{k_j} \right) \qquad \forall j \neq i.
\]
\end{proposition}
\begin{proof}
First of all, by \cite[Theorem 1]{Mimura} and \cite{ConwaySmoller}, we have that the only solution of the system with only the $i$-th density of predator non trivial is the constant solution $W$. The study of the stability of this solution is based on a simple analysis of the linearized system around it: we consider the G\^ateaux differential around $\mathbf{v}$ of the operator describing the system, which is given by
\[
	L(\mathbf{v})[w_1, \dots, w_N, u] = \begin{cases}
		- d_i \Delta w_i - k_i \tilde w_i u + \beta \tilde w_i\sum_{j \neq i} a_{ij} w_j\\
		- d_j \Delta w_j + \left[ k_j \left( \frac{\omega_j}{k_j} - \frac{\omega_i}{k_i} \right)+ \beta\tilde w_i  a_{ji} \right]w_j &\text{for $j \neq i$}\\
		-D \Delta u + \mu \frac{\omega_i}{k_i} u + \omega_i w_i
	\end{cases}
\]
for all $(w_1, \dots, w_N,u) \in \C^{2,\alpha}(\overline{\Omega})$ with homogenous Neumann boundary conditions. To ensure the stability of the solution we need to show that the spectrum of $L$ is contained in $\mathbb{C}^+ = \{z \in \mathbb{C} : \Re(z) > 0 \}$, that is for any $(w_1, \dots, w_N, u) \neq 0$ and $\gamma \in \mathbb{C}$ 
\[
	L(\mathbf{v})[w_1, \dots, w_N, u] = \gamma (w_1, \dots, w_N, u) \implies \Re(\gamma) > 0.
\]
In the previous system, the components corresponding to $j \neq i$ are decoupled from the others, and thus their presence does not influence the stability of $W$. This solution $W$ is stable if and only if
\[
	k_j \left( \frac{\omega_j}{k_j} - \frac{\omega_i}{k_i} \right)+ \beta\tilde w_i  a_{ji} > 0 \qquad \forall j \neq i
\]
which gives the condition established by the proposition; indeed, under this assumption the components $w_j$ with $j \neq i$ are necessarily trivial. Let us show that this condition is enough to ensure the stability: we suppose that the previous system of inequalities holds but there exist $(w_1, \dots, w_N, u) \neq 0$ and  $\gamma \in \mathbb{C}$ with $\Re(\gamma) \leq 0$ solution to
\[
	L(\mathbf{v})[w_1, \dots, w_N, u] = \gamma (w_1, \dots, w_N, u).
\]
Then necessarily $w_j = 0$ for all $j \neq i$, and the system is reduced to
\begin{equation}\label{eqn eigen stab}
	\begin{cases}
		- d_i \Delta w_i  = \gamma w_i + k_i \tilde w_i u \\
		-D \Delta u = -  \omega_i w_i + \left(\gamma - \mu \frac{\omega_i}{k_i} \right) u \\
		\partial_\nu w_i = \partial_\nu u = 0 &\text{ on $\partial \Omega$}
	\end{cases}
\end{equation}
Since any weak solution to the previous system is regular, the stability in $\C^{2,\alpha}(\Omega)$ can be deduced from the solvability of the system in $H^1(\Omega)$. To analyze it, let $\{(\gamma_h,\psi_h)\}_{h \in \N}$ be the spectral resolution of the Laplace operator with homogeneous Neumann boundary condition in $\Omega$ (let us recall that $\gamma_0 = 0$ and $\gamma_h > 0$ for $h > 0$);  since $\{\psi_h\}_{h \in \N}$ is a complete basis of $L^2(\Omega)$, we can write
\[
	w_i = \sum_{h = 0}^\infty a_h \psi_h \quad \text{ and } u = \sum_{h = 0}^\infty b_h \psi_h
\]
as series converging in $L^2(\Omega)$. Inserting these relations in \eqref{eqn eigen stab} and using the orthogonality of the eigenfunctions, we see that the linear system \eqref{eqn eigen stab} is equivalent to the sequence of algebraic eigenvalue problems
\[
	\begin{cases}
		d_i \gamma_h a_h - k_i \tilde w_i b_h =  \gamma a_h \\
		\left(D \gamma_h +  \mu \frac{\omega_i}{k_i} \right) b_h +  \omega_i a_h  = \gamma b_h
	\end{cases} \text{for $h\in \N$.}
\]
By direct inspection, we can observe that $a_h = 0$ if and only if $b_h = 0$. Thus, solving the first equation in $b_h$ and substituting the result in the second, we find that $\gamma$ must be a solution to
\[
	\left(D \gamma_h +  \mu \frac{\omega_i}{k_i} - \gamma \right)\left(d_i \gamma_h -\gamma \right) + k_i \tilde w_i \omega_i = 0
\]
that is
\[
	\gamma = \frac12 \left[ \left( (D+d_i) \gamma_h +  \mu \frac{\omega_i}{k_i}\right) \pm \sqrt{\left( (D+d_i) \gamma_h +  \mu \frac{\omega_i}{k_i}\right)^2 - 4 k_i \tilde w_i \omega_i} \right]
\]
and in particular $\Re(\gamma) > 0$.
\end{proof}

We observe that the diffusion rates do not play any role in the stability of the solutions, while a crucial role is played by the ratio $\omega_i/k_i$. In particular if $i$ is such that
\[
	 \frac{\omega_i}{k_i} < \frac{\omega_j}{k_j}  \qquad \forall j \neq i
\]
then the solution $W$ is asymptotically stable also in a slightly cooperative environment, that is for $\beta < 0$ and small in absolute value. This is a consequence of the fact that the semi-trivial solutions are constant and the different diffusion rates do not play a direct role in the stability of the solution (that is, advantage of low/high diffusion rate). In this setting, the quantity $\omega_i/k_i$ can be interpreted as the fitness of the $i$-th population.

One could then wonder whether the previous stability result is a spurious consequence either of the fact that the simple solutions are constant or of another specific feature of this particular formulation of the system. To clarify this issue, we shall now adapt the proof to a very general framework. Let us consider the following operator
\[%begin{equation}\label{eqn model gen ell}
	\mathcal{S}_{\beta}(\mathbf{v}) := \begin{cases}
		\elle_{i} w_{i}  - \left[f_{i}(x,u,w_i) - \beta \sum_{j \neq i} g_{ij} (x, w_i,w_j)\right] w_i &\text{for all $i \in 
		 \{1,\dots,N\}$}\\
		\elle u  - f(x,u,w_1,\dots,w_N) u
	\end{cases}
\]%end{equation}
defined for $\mathbf{v} = (w_1, \dots, w_N, u)$ in the set
\[
	\mathcal{F}(\Omega) = \left\{ \mathbf{v} \in \C^{2,\alpha}(\overline{\Omega}; \R^{N+1}) : 	\partial_{\nu_i} w_{i} = \partial_\nu u =  0 \text{ on $\partial\Omega$}\right\}
\]
where the respective operators $\elle_i$ and $\elle$ stand for linear strongly elliptic operators of the form
\[
	\elle_i w_i = -\div(A_i(x) \nabla w_i), \quad \elle u = -\div(A(x) \nabla u)
\]
associated with some smooth and uniformly elliptic symmetric matrices $A_i$ and $A$, and the$\nu_i$ and $\nu$ denote the co-normal vector fields associated to the corresponding elliptic operators. We assume in the following that all the terms in the operator $\mathcal{S}_{\beta}$ are sufficiently smooth to justify the following computations, and moreover we suppose that there exists positive constants $ C$ such that for any $\mathbf{v} \in \mathcal{F}(\Omega)$ of non negative components we have
\[	
	\begin{cases}
		f_i(x,u,w_i) \leq C(1+u-w_i)\\
		f(x,u,w_1, \dots, w_N) \leq C(1-u)\\
		g_{ij}(x,w_i,w_j) \geq 0
	\end{cases}
\]
Based on the previous notation, a function $\mathbf{v} \in \mathcal{F}(\Omega)$ is a solution of the generalized model if
\[
	\mathcal{S}_\beta(\mathbf{v}) = 0
\]
while a function $\mathbf{v} \in \C^{1}(\R^+;\mathcal{F}(\Omega)) \cap \C(\overline{\R^+};\mathcal{F}(\Omega))$ is a solution to the parabolic model if
\[
	\begin{cases}
		\partial_t \mathbf{v} + \mathcal{S}_\beta(\mathbf{v}) = 0 &t>0\\
		\mathbf{v}(0)=\mathbf{v}_0 &\mathbf{v}_0 \in \mathcal{F}(\Omega).
	\end{cases}
\]
Using the previous assumptions, we have
\begin{lemma}
For any non-negative initial datum $\mathbf{v}_0 \in \mathcal{F}(\Omega)$ there exists a unique solution $\mathbf{v}$ of the previous parabolic problem. Moreover, there exists $T>0$ and $M>0$, independent of $\beta$, such that 
\[
	0 \leq w_1(t,x), \dots, w_N(t,x), u(t,x) \leq M \qquad \text{for all $t \geq T$, $x \in \overline{\Omega}$}.
\]
If there exist $i \in \{1,\dots, N\}$, $t>0$ and $x_0 \in \overline{\Omega}$ such that $w_i(t,x_0) = 0$ (respectively, $u(t,x_0) = 0$), then $w_i \equiv 0$ (respectively, $u\equiv 0$).
\end{lemma}
\begin{proof}
The proof follows directly from the maximum principle, and thus we omit it (see Lemma \ref{lem existence and bounds} for reasoning of this kind).
\end{proof}
In an analogous fashion, we have a corresponding result for the stationary model. Among the class of all possible solutions, we are interested in the particular case of solutions that have only one component among the first $N$ which is non-trivial.
\begin{definition}\label{def: simple sol}
For a given $i \in \{1, \dots, N\}$, a solution $\mathbf{v} \in \mathcal{F}(\Omega)$ is said to be $i$-simple if $w_j \equiv 0$ for all $j \neq i$ and the other components are positive.
\end{definition}
Let us observe that if $\mathbf{v} \in \mathcal{F}(\Omega)$ is an $i$-simple solution for $\mathcal{S}_\beta$, then it is an $i$-simple solution for any value of $\beta$.

For a given solution $\mathbf{v} \in \mathcal{F}(\Omega)$, let $L(\mathbf{v})$ be the G\^ateaux derivatives of $\mathcal{S}_\beta$ in $\mathcal{F}(\Omega)$, that is for any $\boldsymbol{\varphi} \in \mathcal{F}(\Omega)$:
\begin{multline*}
	L(\mathbf{v})[\boldsymbol{\varphi}] = \lim_{\eps \to 0} \frac{\mathcal{S}_\beta(\mathbf{v}+\eps \boldsymbol{\varphi}) -  \mathcal{S}_\beta(\mathbf{v})}{ \eps} \\
	= \begin{cases}
		\elle_{i} \varphi_{i}  -  \left[f_{i}(x,u,w_i) - \beta \sum_{j \neq i} g_{ij} (x, w_i,w_j)\right] \varphi_i\\
		\qquad - f_{i,u}(x,u,w_i)w_i \varphi - f_{i,w_i}(x,u,w_i)w_i \varphi_i \\
		\qquad + \beta \sum_{j \neq i}g_{ij,w_i}(x,w_i,w_j) w_i \varphi_i + \beta \sum_{j \neq i}g_{ij,w_j}(x,w_i,w_j) w_i \varphi_j\\
		\elle \varphi  - f(x,u,w_1,\dots,w_N) \varphi - f_{,u}(x,u,w_1, \dots, w_N)\varphi\\
		\qquad - \sum_{i=1}^N f_{,w_i}(x,u,w_1, \dots, w_N)\varphi_i
	\end{cases}
\end{multline*}
Analogously, for any fixed $i \in \{1, \dots, N\}$ we define the $i$-th partial derivatives $L_i(\mathbf{v})$ as the  G\^ateaux derivatives of $\mathcal{S}_\beta$ in $\mathcal{F}(\Omega)$ with respect to the direction $\boldsymbol{\varphi} \in \mathcal{F}(\Omega)$ such that $\boldsymbol{\varphi} = (0, \dots, \varphi_i, \dots, 0, \varphi)$, that is
\[
	L_i(\mathbf{v})[\boldsymbol{\varphi}] = \begin{cases}
		\elle_{i} \varphi_{i}  -  \left[f_{i}(x,u,w_i) - \beta \sum_{j \neq i} g_{ij} (x, w_i,w_j)\right] \varphi_i\\
		\qquad - f_{i,u}(x,u,w_i)w_i \varphi - f_{i,w_i}(x,u,w_i)w_i \varphi_i \\
		\qquad + \beta \sum_{j \neq i}g_{ij,w_i}(x,w_i,w_j) w_i \varphi_i\\
		0 \qquad \text{ for $j \neq i$}\\
		\elle \varphi  - f(x,u,w_1,\dots,w_N) \varphi - f_{,u}(x,u,w_1, \dots, w_N)\varphi\\
		\qquad - f_{,w_i}(x,u,w_1, \dots, w_N)\varphi_i
	\end{cases}
\]

Accordingly, we recall that a solution $\mathbf{v} \in \mathcal{F}(\Omega)$ is (strongly) stable if any non-trivial solution $(\gamma, \boldsymbol{\varphi})$ of
\[
	L(\mathbf{v})[\boldsymbol{\varphi}] = \gamma \boldsymbol{\varphi}
\]
has necessarily $\Re(\gamma) > 0$. For $i$-simple solutions we have

\begin{definition}
For a given $i \in \{1,\dots,N\}$, an $i$-simple solution $\mathbf{v} \in \mathcal{F}(\Omega)$ is \emph{one-predator} stable if any non-trivial solution $(\gamma, \boldsymbol{\varphi})$ of
\[
	L_i(\mathbf{v})[\boldsymbol{\varphi}] = \gamma \boldsymbol{\varphi}
\]
with $\boldsymbol{\varphi} = (0, \dots, \varphi_i, \dots, 0, \varphi)$ has necessarily $\Re(\gamma) > 0$.
\end{definition}

An $i$-simple solution is thus one-predator stable if it is stable with respect to all the admissible perturbations that leave unchanged the zero components $w_j$ for $j \neq i$. Clearly, if an $i$-simple solution is stable it is also one-predator stable: under suitable conditions, the inverse is true.
\begin{proposition}\label{prp dock full}
For a given $i \in \{1, \dots, N\}$, let us assume that
\[
	\inf_{x \in \Omega} g_{ji}(x,0,s) > 0 \qquad \text{for all $s > 0$ and $j \neq i$}.
\]
If $\mathbf{v} \in \mathcal{F}(\Omega)$ is an $i$-simple one-predator stable solution $\mathbf{v} \in \mathcal{F}(\Omega)$, then there exists $\bar \beta$ such that $\mathbf{v}$ is a stable solution for all $\beta > \bar \beta$.
\end{proposition}
\begin{proof}
The $i$-simple solution $\mathbf{v} = (0, \dots, w_i, \dots, 0, u)$ is stable if
\[
	\begin{cases}
		\elle_{i} \varphi_{i}  -  \left[f_{i}(x,u,w_i) - \beta \sum_{j \neq i} g_{ij} (x, w_i,0)\right] \varphi_i\\
		\qquad - f_{i,u}(x,u,w_i)w_i \varphi - f_{i,w_i}(x,u,w_i)w_i \varphi_i \\
		\qquad + \beta \sum_{j \neq i}g_{ij,w_i}(x,w_i,0) w_i \varphi_i + \beta \sum_{j \neq i}g_{ij,w_j}(x,w_i,0) w_i \varphi_j = \lambda \varphi_i\\
		\elle_{j} \varphi_{j}  -  \left[f_{j}(x,u,0) - \beta g_{ji} (x, 0,w_i)\right] \varphi_j = \lambda \varphi_j\\
		\elle \varphi  - f(x,u,0,\dots,w_i, \dots, 0) \varphi - f_{,u}(x,u,0,\dots,w_i, \dots, 0)\varphi\\
		\qquad - \sum_{i=1}^N f_{,w_i}(x,u,0,\dots,w_i, \dots, 0)\varphi_i = \lambda \varphi
	\end{cases}
\]
has a nontrivial solution $\boldsymbol{\varphi} \in \mathcal{F}(\Omega)$ if and only if $\Re(\lambda) > 0$. Let us consider the equations of index $j \neq i$, which are decoupled from the other equations in the system. They read:
\[
	\begin{cases}
		\elle_{j} \varphi_{j}  =  \left[f_{j}(x,u,0) - \beta g_{jh} (x, 0,w_i) + \lambda \right] \varphi_j &\text{in $\Omega$}\\
		\partial^{\mathcal{L}_j}_{\nu} \varphi_j = 0 &\text{on $\partial \Omega$}
	\end{cases}
\]
Since the operator $\elle_j$ is self-adjoint\footnote{More precisely, the operator is self-adjoint if seen as an operator acting on $H^1(\Omega)$ functions, and the conclusion can be reached using the regularity assumptions on its coefficients.}, any non-trivial solution of the system must have $\lambda \in \R$. The solution $\mathbf{v}$ being an $i$-simple solution, by the maximum principle it follows that 
\[
	\inf_{x\in\Omega} w_i(x) = c > 0.
\]
As a result, thanks to our assumptions, there exists $\bar \beta \geq 0$ such that
\[
	\bar \beta \geq  \sup_{x \in \Omega} \frac{f_{j}(x,u,0)}{g_{ji} (x, 0,w_i)}  \qquad \text{for all $j \neq i$.}
\]
Choosing $\beta > \bar \beta$ and testing the equation in $\varphi_j$ by $\varphi_j$ itself, we obtain
\[
	\int_{\Omega} A_j(x) \nabla w_j \cdot \nabla w_j =  \int_{\Omega} \left[f_{j}(x,u,0) - \beta g_{ji} (x, 0,w_i) + \lambda \right] \varphi_j^2 < \lambda \int_{\Omega} \varphi_j^2 
\]
thus either $\lambda > 0$ or the component $\varphi_j = 0$. On the other hand, assuming that $\Re(\lambda) \leq 0$, we find a contradiction with the internal stability of the solution $\mathbf{v}$.
\end{proof}

\section{Existence of non homogeneous solutions: a bifurcation analysis}\label{sec bif}

We continue the investigation of the asymptotic properties of the system \eqref{eqn model}, by now studying the set of solutions of the corresponding (stationary) elliptic problem. We consider here the model \eqref{eqn model} under the assumption that the domain $\Omega$ is occupied by only two groups of predators, having the same parameters. This system reads:
\[
	\begin{cases}
		- d \Delta w_1 = \left( - \omega + k u - \beta w_2\right) w_1 &\text{ in $\Omega$}\\
		- d \Delta w_2 = \left( - \omega + k u - \beta w_1\right) w_2  &\text{ in $\Omega$}\\
		- D \Delta u = \left(\lambda - \mu u - k (w_1+w_2)\right)u &\text{ in $\Omega$}\\
		\partial_\nu w_i = \partial_\nu u = 0 &\text{ on $\partial \Omega$}
	\end{cases}
\]
for which we look for solutions $(w_1,w_2,u) \in \C^{2,\alpha}(\overline{\Omega})$. Let us point out that here we take $\mu_1 = \mu_2 = 0$. Alternatively, we can  easily generalize the results that will we show in the following to the case of positive saturation coefficients (though the computations are inevitably more involved). Since we are looking for stationary solutions, the system can be simplified by some linear substitutions. Indeed, letting
\[
	u \mapsto \frac{d}{D} u, \; \lambda \mapsto \lambda D, \;  \mu \mapsto \mu \frac{D^2}{d}, \;  k \mapsto k D, \;  \omega \mapsto \omega d, \;  \beta \mapsto \beta d
\]
we can reformulate the system as
\begin{equation}\label{eqn model red}
	\begin{cases}
		- \Delta w_1 = \left( - \omega + k u - \beta w_2\right) w_1 &\text{ in $\Omega$}\\
		- \Delta w_2 = \left( - \omega + k u - \beta w_1\right) w_2  &\text{ in $\Omega$}\\
		- \Delta u = \left(\lambda - \mu u - k (w_1+w_2)\right)u &\text{ in $\Omega$}\\
		\partial_\nu w_i = \partial_\nu u = 0 &\text{ on $\partial \Omega$}
	\end{cases}
\end{equation}
We recall the definition of the set
\[
	\mathcal{F}(\Omega) := \left\{(w_1,w_2,u) \in \C^{2,\alpha}(\overline{\Omega}) : \partial_\nu w_1 = \partial_\nu w_2 = \partial_\nu u = 0 \text{ on $\partial \Omega$}\right\}.
\]
We are interested in non negative solutions of the system. Letting all the other parameters of the model fixed, we shall study the set of the solutions of \eqref{eqn model red} by varying the competition strength $\beta$. Let us recall that the assumption (H) holds, that is, in this context, $\lambda k > \mu \omega$.

We start by recalling a result concerning the regularity of solutions of system \eqref{eqn model red}. This result follows from Lemma \ref{lem existence and bounds}, but we report it here for the reader's convenience.

\begin{lemma}\label{lem reg ell}
Let $(w_1,w_2, u) \in H^1(\Omega)$ be a non negative weak solution to \eqref{eqn model red}. Then
\begin{itemize}
	\item the solutions are classical. More precisely,  $(w_1,w_2,u) \in \C^\infty(\Omega) \cap \C^{2,\alpha}(\overline{\Omega})$ for any $\alpha < 1$ and the regularity is limited only by that of $\Omega$;
	\item $(w_1,w_2,u)$ are non negative and bounded uniformly in $\beta$, that is 		
	\begin{equation}\label{eqn upper bound 2}
		\begin{cases}
			w_1 \geq 0, w_2 \geq 0, 0 \leq u \leq \lambda/ \mu \\
			u + w_1 + w_2 \leq \displaystyle \frac{( \lambda + \omega) \lambda}{\mu \omega}
		\end{cases}
	\end{equation}
	and either all the inequalities are strict or the solution is constant;
\end{itemize}
\end{lemma}

\begin{proof}
All the assertions in the statement are rather straightforward consequences of the maximum principle and the classical regularity theory of elliptic equations. We only observe that the last inequality follows by summing the three equations together. This yields to
\[
  -\Delta (u + w_1 + w_2) \leq (\lambda + \omega) u - \omega (u + w_1 + w_2) \leq (\lambda + \omega) \frac{\lambda}{\mu} - \omega (u + w_1 + w_2).
\]
We conclude again by virtue of the maximum principle.
\end{proof}

Lemma \ref{lem reg ell} gives a description of the solutions of the system \eqref{eqn model red}, but it contains no information about the existence of such solutions. In the following, our aim is to complete this aspect, showing that the set of solutions is rich. Before doing so, we need to introduce some notation. 

For a given solution $(w_1,w_2,u) \in \mathcal{F}(\Omega)$ of the system \eqref{eqn model red}, the G\^ateaux derivate in $\mathcal{F}(\Omega)$ associated to \eqref{eqn model red} computed at  $(w_1,w_2,u) $ is given by
\[
	L_\beta \boldsymbol{\varphi} = -\Delta \boldsymbol{\varphi} - A_\beta \boldsymbol{\varphi}, \quad \text{for any $\boldsymbol{\varphi} \in \mathcal{F}(\Omega)$}
\]
where $A_\beta = A_\beta(w_1,w_2,u) \in \C^{2,\alpha}(\overline{\Omega}, \R^{3\times3})$ is
\[
	A = A_\beta = \left( \begin{array}{ccc}
              -\omega + k u - \beta w_2 & - \beta w_1  & k w_1   \\
              -\beta w_2 & -\omega + k u - \beta w_1  &  k w_2  \\
             -k u & - k u & \lambda - 2 \mu u - k w_1 - k w_2  
          \end{array} \right).
\]
The solution $(w_1,w_2,u)$ is said to be (strongly linearly) stable if any non-trivial solution $(\gamma, \boldsymbol{\varphi})$ of the linearized equation
\[
	L_\beta  \boldsymbol{\varphi} = \gamma \boldsymbol{\varphi} 
\]
has necessarily $\Re(\gamma) > 0$ and weakly stable we can only infer that $\Re(\gamma) \geq 0$. It is said to be (strongly linearly) unstable if, on the contrary, there exists a non-trivial solution with $\Re(\gamma) < 0$.

If the solution $(w_1,w_2,u)$ in the previous definition is constant, its stability can be directly deduced from the spectrum of the matrix $A_\beta$ or, more explicitly, from that of $-A$. We start with the simplest scenario, that is the limit case $\beta = 0$. Under this assumption, since the densities of predators do not interact directly with each other, we can simplify drastically the system and give a complete description of the set of solutions of the system.
\begin{lemma}\label{lem trivial branch of sol}
The unique non negative solutions $(w_1, w_2, u)$ to the system \eqref{eqn model red} with $\beta = 0$ are the two unstable constant solutions
\[
	 (0,0,0), \left(0,0,\frac{\lambda}{\mu}\right)
\]
and the one-parameter family of (weakly) stable ones
\[
	 s \in [0,1] \mapsto \left(\frac{\lambda k - \mu \omega}{k^2} s, \frac{\lambda k - \mu \omega}{k^2} (1-s),\frac{\omega}{k}\right).
\]
\end{lemma}
\begin{proof}
In this proof, we shall only classify the solutions; the study of the stability will be postponed until Lemma \ref{lem matrix A}, where we shall address the question about stability of constant solutions for $\beta \geq 0$ more generally.

Since for $\beta = 0$ the densities of predators do not interact directly with each other, we can simplify the system introducing the new variable $V = w_1 + w_2$, which, together with $u$ is a solution of the classical (i.e.\ one predator) Lotka-Volterra system
\[
	\begin{cases}
		- \Delta V = - \omega V + k V u &\text{ in $\Omega$}\\
		- \Delta u = \lambda u - \mu u^2 - k V u &\text{ in $\Omega$}\\
		\partial_\nu V = \partial_\nu u = 0 &\text{ on $\partial \Omega$}.
	\end{cases}
\]
By a classical result of Mimura \cite[Theorem 1]{Mimura} it follows that the previous system has only constant solutions, that is solutions of the algebraic system
\[
	\begin{cases}
		(k u - \omega )V =0\\
		(\lambda - \mu u - k V )u = 0
	\end{cases}
\]
When $V = 0$, we have the solutions $u = 0$ or $u = \lambda / \mu$ which correspond to the first two solutions in the statement (recall that $w_1$ and $w_2$ are non negative, that is, in this case, $w_1 = w_2 = 0$). On the other hand, if $u = \omega / k$, we obtain the solution $V = w_1 + w_2 =  (\lambda k - \mu \omega)/ k^2$. Substituting this information in \eqref{eqn model red} we obtain that both $w_1$ and $w_2$ are harmonic functions, hence constants. 
\end{proof}

As we shall see later, the value $\beta = 0$ corresponds to a bifurcation point of multiplicity one for the system \eqref{eqn model red} around the solution
\[
	(w_1,w_2,u) =  \left(\frac{\lambda k - \mu \omega}{2k^2}, \frac{\lambda k - \mu \omega}{2k^2},\frac{\omega}{k}\right),
\]
so that the one-parameter family of solutions of Lemma \ref{lem trivial branch of sol} is nothing but the branch of solutions emanating from it.

\begin{lemma}\label{lem matrix A}
When $\beta > 0$, system \eqref{eqn model red} admits four different types of constant solutions:
\begin{enumerate}[ label = $(\alph*)$]
	\item the solution $(0,0,0)$, which is strongly unstable;
	\item the solution
	\[
	w_1 = 0, \, w_2 = 0, \,u = \frac{\lambda}{k}
	\]
	which is strongly unstable;
	\item the solutions 
	\[
	w_1 = \frac{\lambda k - \mu \omega}{k^2}, \, w_2 = 0, \, u = \frac{\omega}{k} \quad \text{ and } w_1 = 0, \, w_2 = \frac{\lambda k - \mu \omega}{k^2}, \, u = \frac{\omega}{k}
	\]
	which are strongly stable;
	\item the family of solutions
	\[
		w_1 = w_2 = \frac{\lambda k - \mu \omega}{\mu \beta + 2 k^2}, \; u = \frac{\lambda \beta + 2 k \omega}{\mu \beta + 2 k^2}
	\]
	which are unstable for $\beta > 0$. In particular, in this latter case,
	\[
		\sigma(A_\beta) = \left\{ \beta\frac{\lambda k - \mu \omega}{\mu \beta + 2 k^2}, \gamma_{1,\beta}, \gamma_{2,\beta}\right\}
	\]
	where $\gamma_{1,\beta}$ and $\gamma_{2,\beta}$ are two, possibly complex conjugate, eigenvalues with negative real part.
\end{enumerate}
\end{lemma}
Later we will prove that the solutions in Lemma \ref{lem matrix A} are the only solution of \eqref{eqn model red} when $\beta > 0$ is sufficiently small (compare with Proposition \ref{prop beta pos}).
\begin{proof}
The proof is a rather straightforward computation, but we include it here in order to glean from it an interpretation of the results.

The solution $(0,0,0)$ corresponds to the matrix
\[
	A_\beta = \left( \begin{array}{ccc}
              -\omega & 0  & 0  \\
              0 & -\omega & 0  \\
             0 & 0 & \lambda
          \end{array} \right)
\]
which is already in a diagonal form. The instability of this solution is caused by the eigenvalues $-\lambda < 0$ of $-A$, which corresponds to the constant eigenfunction $(0,0,1)$. As a result, in complete accordance with other biological models, it implies that a logistic growth law in the prey population is responsible for an (initial) exponential growth, uniform in all the domain $\Omega$, at least when the population is small. Let us observe that none of the spectral and stability properties of the trivial solution depends on the competition $\beta$.

Similar computations hold for the solution $(0, 0, \lambda/\mu)$, whose associated matrix is
\[
	A_\beta = \left( \begin{array}{ccc}
              \frac{\lambda k - \mu \omega}{\mu} & 0  & 0  \\
              0 & \frac{\lambda k - \mu \omega}{\mu} & 0  \\
             - \lambda k / \mu & - \lambda k / \mu & -\lambda
          \end{array} \right).
\]
Here the eigenvalues are $\frac{\lambda k - \mu \omega}{\mu}$ (of multiplicity 2) and $-\lambda$. The eigenspaces of the matrix is generated by the vectors
\[
	\left( \begin{array}{c}
              1\\ 0 \\ -\frac{\lambda k}{\lambda k - \mu \omega + \lambda \mu}
              \end{array} \right), \left( \begin{array}{c}
              0 \\ 1 \\ -\frac{\lambda k}{\lambda k - \mu \omega + \lambda \mu}
              \end{array} \right)  \quad \text{ and } \quad
              \left( \begin{array}{c}
              0\\ 0 \\ 1
              \end{array} \right)
\]
respectively.

To discuss the solutions of type (c), consider for example the solution $w_1 =(\lambda k - \mu \omega)/k^2$, $w_2 = 0$ and $u = \omega/k$. In this case the matrix $A_\beta$ is
\[
	A_\beta = \left( \begin{array}{ccc}
              0 & - \beta w_1  & k w_1   \\
              0 & - \beta w_1  & 0 \\
             -\omega & - \omega & - \mu \omega/k
          \end{array} \right).
\]
Thus, the spectrum of $A_\beta$ consists of
\[
	- \beta w_1,\; - \frac{\mu \omega/k \pm \sqrt{(\mu \omega/k)^2-4k \omega w_1}}{2},
\]
which implies strong stability of these solutions. As already observed in the previous section, this result in unchanged even when the two populations of predators have different parameters, as long as $\beta > 0$ (compare Proposition \ref{prop dock new}).

In the case of the constant solutions of type (d), recalling that here $w_1 = w_2$, the matrix $A_\beta$ reduces to
\[
	A_\beta = \left( \begin{array}{ccc}
              0 & - \beta w_1  & k w_1   \\
              -\beta w_1 & 0  &  k w_1  \\
             -k u & - k u & - \mu u
          \end{array} \right).
\]
By direct inspection, we see that $\beta w_1 \geq 0$ is an eigenvalue, implying in particular that these solutions are unstable for $\beta > 0$. Using this information, we can factorize the characteristic polynomial of $A_\beta$, yielding
\[
	\det(A-\gamma \id) = (\gamma - \beta w_1) \left[\gamma^2 + (\beta w_1 + \mu u) \gamma + (2k^2 u w_1 + \beta \mu u w_1)\right]
\]
and the spectrum of $A_\beta$ consists of
\[
	\beta w_1, \; - \frac{(\beta w_1 + \mu u) \pm \sqrt{(\beta w_1 + \mu u)^2-(2k^2 u w_1 + \beta \mu u w_1)}}{2},
\]
and this concludes the proof.
\end{proof}

The set of non-trivial constant solutions undergoes a transformation as $\beta$ changes from $\beta = 0$ to $\beta > 0$, see Lemma \ref{lem trivial branch of sol}. Moreover, the spectrum of the matrix $A_0$, computed on the linear set of solutions is given by
\[
	0, \; - \frac{\mu \omega/k \pm \sqrt{(\mu \omega/k)^2-4k \omega (w_1+w_2)}}{2}.
\]
The zero eigenvalue underlines the degeneracy of the constant solutions, as they form a linear subspace, while the other two strictly negative eigenvalues confirm that this set of solutions is stable with respect to perturbations that move away from this configuration, i.e.~non homogeneous perturbation (see Proposition \ref{prp conway}).

\begin{remark}\label{rem stab indip beta}
The stability of the solutions belonging to the classes $(a)$, $(b)$ and $(c)$ does not depend on $\beta$. More precisely, in the classes $(a)$ and $(b)$ the spectrum of $A_\beta$ is independent of $\beta$, while in the third case $(c)$ the spectrum is also contained in $\mathbb{C}^- := \{z\in\mathbb{C}: \Re(z)<0\}$.
\end{remark}

We can say more about constant solution, and in particular we have that if a component is constant, so are the other.

\begin{lemma}\label{lem u constant}
For a solution $(w_1,w_2,u)$ of \eqref{eqn model red}, if one component is constant, then also all the other components are constant.
\end{lemma}
\begin{proof}
The case for $\beta = 0$ is already considered in Lemma  \ref{lem trivial branch of sol}. Thus we can assume $\beta > 0$. We start by assuming that $u$ is constant. If $u$ is zero, we conclude directly by the maximum principle. Assuming that $u$ is a positive constant, from the equation in $u$, we find that necessarily
\[
	w_1 + w_2 = \frac{\lambda- \mu u}{k}
\]
is a non-negative constant. This yields in the equation for $w_i$, $i =1,2$:
\[
	\begin{cases}
		- \Delta w_i  = \left(- \omega + k u + \beta \frac{\mu}{k} u - \beta\frac{\lambda}{k} + \beta w_i\right) w_i \\
		\partial_\nu w_i = 0 &\text{ on $\partial \Omega$}.
	\end{cases}
\]
Summing up the two equations, we obtain moreover
\[
	w_1^2+w_2^2 = \frac{\lambda-\mu u}{k} \left(\frac{\lambda-\mu u}{k} - \frac{k u- \omega}{\beta} \right) \geq 0
\]
As a result, we have obtained the identities
\[
	w_1 + w_2 = a, \qquad w_1^2+w_2^2 = b
\]
for some non-negative constant $a$ and $b$, which directly implies that $w_1$ and $w_2$ are constant. Using this information, it is also possible to compute explicitly the solutions, and in particular we find $u = \omega/k$.

We now assume that $w_1$ is constant. From the equation in $w_1$ we find that $w_1 = 0$ or $\beta w_2 = k u - \omega$. The former case is equivalent to assuming $\beta = 0$. We then need only to address the latter. Substituting the previous identity in the equation in $u$ we find
\[
	\begin{cases}
		- \Delta u  = \left(A - B u \right) u &\text{ in $\Omega$}\\
		\partial_\nu u = 0 &\text{ on $\partial \Omega$}.
	\end{cases}
\]
where $A$ is a real constant and $B$ is a strictly positive constant. If $A \leq 0$ then $u$ is zero. On the other hand, if $A > 0$ by the maximum principle (see Lemma \ref{lem max neumann} below), we find that $u$ is again a constant, and we can conclude as above.
\end{proof}

Note that the arguments of the proof are only valid in the case of two predators $w_1$, $w_2$.
In the previous result, we made use of the following classical consequence of the maximum principle.
\begin{lemma}\label{lem max neumann}
Let $\Omega \subset \R^n$ a bounded smooth domain, $A$ and $B$ positive constants. If $u \in H^1(\Omega)$ is a non negative solution to
\[
	\begin{cases}
		- \Delta u = (A - B u) u &\text{ in $\Omega$}\\
		\partial_\nu u = 0 &\text{ on $\partial \Omega$}
	\end{cases}
\]
then $u \equiv 0$ or $u \equiv A/B$.
\end{lemma}

Here, we are mostly interested in solutions which are not homogeneous in space (i.e., non constant solutions). We will derive their existence through bifurcation arguments. 

First, we introduce some notation.

\begin{definition}
We denote with $\mathscr{P} \subset \R^+ \times \mathcal{F}(\Omega)$ the set of all solutions $(\beta, w_{1,\beta},w_{2,\beta}, u_{\beta})$ of \eqref{eqn model red} with competition parameter $\beta >0$ such that all of its components are strictly positive. Let also $\mathscr{S}_0$ stand for the set of constant solutions $(\beta, w_1,w_2,u)$ of the form
\[
	w_1 = w_2 = \frac{\lambda k - \mu \omega}{\mu \beta + 2 k^2}, u = \frac{\lambda \beta + 2 k \omega}{\mu \beta + 2 k^2}
\]
for all values of $\beta > 0$. Lastly, we let
\[
	\mathscr{S}_1 = \mathscr{P} \setminus \mathscr{S}_0
\]
and $\mathscr{S}_* = \overline{\mathscr{S}_1}$, where the closure is taken in the $\R \times \mathcal{F}(\Omega)$ topology.
\end{definition}

Observe that the solutions in $\mathscr{S}_0$ are parameterized in $\beta$. We start with the asymptotic analysis when $\beta \to \infty$ of the spectrum associated to solutions of type $\mathscr{S}_0$.
\begin{lemma}\label{lem eigen A on S}
Let $(w_1, w_2, u) \in \mathscr{S}_0$. The eigenvalues of $A_\beta$ behave like
\[
	\beta\frac{\lambda k - \mu \omega}{\mu \beta + 2 k^2} \sim \frac{\lambda k - \mu \omega}{\mu}, \quad \gamma_{1,\beta} \sim - \left(\lambda + \frac{\lambda k - \mu \omega}{\mu}\right), \quad \gamma_{2,\beta} \to 0^- \qquad \text{as $\beta \to \infty$.}
\]
As a consequence, the supremum of the spectrum of the matrix $A_\beta$ is described, in terms of $\beta$, by the curve
\[
	\beta \mapsto \beta\frac{\lambda k - \mu \omega}{\mu \beta + 2 k^2}.
\]
Moreover, this supremum of the spectrum is monotone increasing in $\beta$ and its limit as $\beta \to \infty$ can be made arbitrarily large by taking $\mu$ small accordingly. In particular, in the limit case $\mu = 0$, the spectrum is unbounded.

Lastly, the unstable direction of $A_\beta$ is spanned by the eigenvector $(1, -1, 0)$.
\end{lemma}

We can now derive a result concerning the existence of non constant solutions: our construction is implicit and uses the topological degree argument through a bifurcation analysis of the set of constant solutions. Let $0 = \gamma_0 < \gamma_1 \leq \gamma_2 \leq \dots$ denoted the unbounded sequence of eigenvalues of the Laplace operator with homogeneous Neumann boundary conditions and let $\{\psi_i\}$ be the corresponding eigenfunctions:
\begin{equation}\label{eqn eigen lap neu}
	\begin{cases}
		-\Delta \psi_i = \gamma_i \psi_i &\text{in $\Omega$}\\
		\partial_\nu \psi_i = 0 &\text{on $\partial \Omega$}.
	\end{cases}
\end{equation}
We define $n^* \in \N$ to be the largest index corresponding to an eigenvalue $\gamma_{n^*}$ such that 
\[
	\gamma_{n^*} < \frac{\lambda k - \mu \omega}{\mu}.
\]
We assume in the following that $n^* \geq 1$. Let also $\beta_n > 0$ be defined by
\[
	\beta_n \frac{\lambda k - \mu \omega}{\mu \beta_n + 2 k^2} = \gamma_n.
\]
Observe that $n^*$ can be made as large as desired by taking $\mu$ small accordingly.
\begin{theorem}\label{thm bif}
For any $1 \leq n \leq n^*$, if $\gamma_n$ is an eigenvalue of odd multiplicity, then $(\beta_n, w_{1,n}, w_{2,n}, u_n)$ is a bifurcation point from the branch of solutions $\mathscr{S}_0$ into non constant solutions $\mathscr{S}_*$. More precisely, there exists a maximal closed and connected subset $\mathscr{C}_n \subset \mathscr{S}_*$ of solutions of \eqref{eqn model red} such that $\mathscr{C}_n$ contains the point $(\beta_n, w_{1,n}, w_{2,n}, u_n)$ and either
\begin{itemize}
	\item $\mathscr{C}_n$ is unbounded in $\beta$, or
	\item $\mathscr{C}_n$ contains another point $(\beta_m, w_{1,m}, w_{2,m}, u_m)$ for a different value of $m$.
\end{itemize}

Aside from these bifurcation points emanating from the branch of solutions $\mathscr{S}_0$, the set $\mathscr{C}_n$ consists of solutions $(w_1, w_2, u)$ which are non constant.
\end{theorem}

\begin{remark}\label{rem bif from 0}
One could wonder what happens for the eigenvalue $\gamma_0 = 0$. Actually, this is already contained in the previous remarks: indeed, $\gamma_0$ corresponds to the value $\beta = 0$, and we have already observed in Lemma \ref{lem trivial branch of sol} that in this situation there exists a one dimensional subspace of constant solutions emanating from this point. Thus it is also a bifurcation point from $\mathscr{S}_0$ (in such case, the branch is explicit and the solutions are constant). This point is particular and indeed we do not take it into account in the statement of Theorem \ref{thm bif}.
\end{remark}

\begin{remark}\label{rem bif sim}
If one assumes some symmetry properties for the domain $\Omega$, one can then also give a more detailed description of the branches in Theorem \ref{thm bif}. In particular we can show that the symmetries of the eigenfunctions are preserved along a global branch of solutions, see for instance \cite{Ruelle,Sattinger}.
\end{remark}

\begin{proof}
The theorem follows from the classical bifurcation theorem of Rabinowitz, see \cite{Rabinowitz_JFA, RabBere}. For $\beta > 0$ and a corresponding nontrivial constant solution $(w_1,w_2,u)$ with $w_1 = w_2$, we look for a new solution of the form $(w_1 + \varphi_1, w_2 + \varphi_2, u + \varphi)$, for small perturbations $\boldsymbol{\varphi} = (\varphi_1, \varphi_2, \varphi) \in \mathcal{F}(\Omega)$. Inserting this ansatz  in the system \eqref{eqn model red} we obtain
\begin{equation}\label{eqn rem bif}
		-\Delta \boldsymbol{\varphi} = A_{\beta}  \boldsymbol{\varphi} +  \left( \begin{array}{c}
	              k \varphi_1 \varphi - \beta \varphi_1 \varphi_2\\
        		      k \varphi_2 \varphi - \beta \varphi_1 \varphi_2\\
		      -\mu \varphi^2 - k (\varphi_1+\varphi_2)\varphi
	          \end{array} \right) = A_\beta \boldsymbol{\varphi} + H(\beta,\boldsymbol{\varphi}) \qquad \text{in $\Omega$}
\end{equation}
completed by homogeneous Neumann boundary conditions. Here the nonlinear functional $H: (\R,\mathcal{F}(\Omega)) \to \mathcal{F}(\Omega)$ is continuous and $\|H(\beta, \boldsymbol{\varphi})\|_{\mathcal{F}(\Omega)} \leq C \|\boldsymbol{\varphi} \|_{\mathcal{F}(\Omega)}^2$ for a constant $C>0$ that can be chosen uniformly on compact sets of $\beta \in [0, +\infty)$. Let us now introduce the operator $L \in \mathcal{K}(\mathcal{F}(\Omega);\mathcal{F}(\Omega))$ defined as the linear map such that for any $\mathbf{u}, \mathbf{f} \in \mathcal{F}(\Omega)$
\[
	\mathbf{u} = L \mathbf{f} \Leftrightarrow \begin{cases}
		-\Delta \mathbf{u} + \mathbf{u} = \mathbf{f}   &\text{in $\Omega$}\\
		\partial_\nu \mathbf{u} = 0 &\text{on $\Omega$}.
	\end{cases}
\]
We can rewrite the perturbed system as
\begin{equation}\label{eqn bif rab}
	\boldsymbol{\varphi} = (A_\beta + \id) L \boldsymbol{\varphi} + LH(\beta,\boldsymbol{\varphi}) =  (A_\beta + \id) L \boldsymbol{\varphi} + h(\beta,\boldsymbol{\varphi})
\end{equation}
where now $h : (\R,\mathcal{F}(\Omega)) \to \mathcal{F}(\Omega)$ is a compact operator. Furthermore, it is such that $\|h(\beta, \boldsymbol{\varphi})\|_{\mathcal{F}(\Omega)} \leq C \|\boldsymbol{\varphi} \|_{\mathcal{F}(\Omega)}^2$  with a constant $C>0$ that again can be chosen uniformly on compact sets of $\beta$. We are now in a position to apply the global bifurcation theorem of Rabinowitz \cite{Rabinowitz_JFA, RabBere}. Indeed, as $\beta$ varies, $(A_\beta + \id) L $ is a homotopy of compact operators. It is known that a sufficient condition for a value $\bar \beta$ to be a bifurcation point for the equation \eqref{eqn bif rab} is that the set of solutions to the linear equation
\[
	\boldsymbol{\varphi} = (A_{\bar \beta} + \id) L \boldsymbol{\varphi} 
\]
has odd dimension. This equation translates into the $3$-component system
\[
	\begin{cases}
		-\Delta \boldsymbol{\varphi} = A_{\bar \beta}  \boldsymbol{\varphi} &\text{in $\Omega$}\\
		\partial_\nu \boldsymbol{\varphi} = 0 &\text{on $\Omega$}.
	\end{cases}
\]
We have already studied the spectral properties of the matrix $A_{\bar \beta}$ in Lemma \ref{lem matrix A}. The matrix has a unique positive eigenvalue $\bar\beta (\lambda k - \mu \omega)/(\mu \bar\beta + 2 k^2)$ that correspond to the eigenvector $(-1,1,0)$. As a consequence, $(\gamma_i, \psi_i)$ is an eigenvalue-eigenvector couple of \eqref{eqn eigen lap neu} and $\bar\beta (\lambda k - \mu \omega)/(\mu \bar\beta + 2 k^2) = \gamma_i$ if and only if $\boldsymbol{\varphi} = (\psi_i, - \psi_i, 0)$ solves the previous system for the prescribed value of $\bar \beta$. In particular if $\gamma_i$ has odd multiplicity, then $\bar \beta$ is a bifurcation point in the sense of the theorem.

It remains to show that the continua $\mathscr{C}_n$ are either unbounded in $\beta$ or meet the set $\mathscr{S}_0$ in another bifurcation point: recalling the global bifurcation theorem of Rabinowitz \cite[Theorem 1.3]{Rabinowitz_JFA}, we already know that each continuum is either unbounded in $\R \times \mathcal{F}(\Omega)$ or touches the set $\mathscr{S}_0$ at an other bifurcation point. Hence the statement in the theorem is reduced to showing that if the continuum $\mathscr{C}_n$ is unbounded, then it must be unbounded in $\beta$. We recall that, by Lemma \ref{lem reg ell}, the non-negative solutions satisfy the inequalities 
\[
	w_1 \geq 0, \quad w_2 \geq 0, \quad 0 \leq u \leq \lambda/ \mu, \quad u + w_1 + w_2 \leq \frac{( \lambda + \omega)^2}{4 \mu \omega}
\]
and either all the inequalities are strict or the solution is constant. It follows that if $\beta$ is bounded on $\mathscr{C}_n$, there must exists on $\mathscr{C}_n$ a solution which is constant. In view of Lemma \ref{lem matrix A}, discarding the solutions on $\mathscr{S}_0$, the only possibilities are solutions which are either case $(a)$ and $(b)$ strongly unstable or case $(c)$ strongly stable. Recall that these properties do not depend on $\beta >0$. We shall exclude these possibilities in the following results.
\end{proof}

\begin{lemma}\label{lem 000 isolated}
The set of solutions $\R \times (0,0,0)$ is isolated in $\mathscr{P}$ for $\beta$ bounded.
\end{lemma}
\begin{proof}
We assume that there exists a sequence $(\beta_n, w_{1,n}, w_{2,n}, u_{n}) \in \mathscr{P} \setminus \R \times (0,0,0)$ of solutions of \eqref{eqn model red} such that $\beta_n > 0$ and, as $n\to+\infty$, we have $\beta_n \to \bar \beta \in [0,+\infty)$ and $\mathbf{v}_n := (w_{1,n}, w_{2,n}, u_{n}) \to (0,0,0)$ in $\mathcal{F}(\Omega)$. We can rewrite \eqref{eqn model red} as
\[
		-\Delta \mathbf{v}_n =  \left( \begin{array}{ccc}
	              -\omega & 0 & 0\\
        		 0 &-\omega & 0\\
		      0 & 0 & \lambda
	          \end{array} \right)  \mathbf{v}_n +  \left( \begin{array}{c}
	              k u_n w_{1,n} - \beta_n w_{1,n} w_{2,n}\\
        		 k w_{2,n} u_n - \beta_n w_{1,n} w_{2,n}\\
		      -\mu u_n^2 - k (w_{1,n}+w_{2,n}) u_{n}
	          \end{array} \right) = A_{0,\beta} \mathbf{v}_{n} + H_0(\beta_n,\mathbf{v}_{n}) 
\]
in $\Omega$, where  $H_0: (\R,\mathcal{F}(\Omega)) \to \mathcal{F}(\Omega)$ is continuous and $\|H_0(\beta, \mathbf{v})\|_{\mathcal{F}(\Omega)} \leq C \|\mathbf{v} \|_{\mathcal{F}(\Omega)}^2$ locally at $\beta = \bar \beta$. We follow a reasoning similar to that of Theorem \ref{thm bif}. We recall that the eigenvalues of the Laplacian with Neumann boundary conditions are non negative. Thus, by the stability analysis of Lemma \ref{lem matrix A}, for $n \to +\infty$ we have that there exists $\eps_n \to 0$ and an eigenpair  $(\gamma_i, \psi_i)$ of \eqref{eqn eigen lap neu} such that
\[
	\lambda = \gamma_i \qquad \text{and} \qquad \mathbf{v}_n = \eps_n  \left( \begin{array}{c} 0 \\ 0 \\ 1\end{array} \right) \psi_i  + o(\eps_n).
\]
Since $\lambda > 0$, it must be that the index $i$ is strictly positive. But then the eigenfunction $\psi_i$ changes sing in $\Omega$, and for $n$ sufficiently large, so does $u_n$. We reach the desired contradiction, as we are considering only solutions that are non negative in $\Omega$.
\end{proof}

In Lemma \ref{lem asymp positive sol} we will show that the same conclusion holds for $\beta_n \to +\infty$, that is, the sets $\R \times (0,0,0)$ and $\mathscr{P} \setminus \R \times (0,0,0)$ are at a positive distance. 

We now turn to the other line of constant solutions. 

\begin{lemma}\label{lem 001 isolated}
The set of solutions $\R \times (0,0,\lambda/\mu)$ is isolated in $\mathscr{P}$ for $\beta$ bounded.
\end{lemma}
\begin{proof}
The proof is rather similar to that of Lemma \ref{lem 000 isolated}. We omit the details. We only point out that this time the conclusion is reached exploiting the expansion of the solutions as
\[
	\frac{\lambda k - \mu \omega}{\mu} = \gamma_i \qquad \text{and} \qquad \mathbf{v}_n = \eps_n  \left( \begin{array}{c} 1 \\ 0 \\ -\frac{\lambda k}{\lambda k - \mu \omega + \lambda \mu} \end{array} \right) \psi_i + o(\eps_n)
\]
and the assumption (H), that implies again $\gamma_i > 0$.
\end{proof}

Observe that here, in contrast with the case of $\R \times (0,0,0)$, the sets $\R \times (0,0,\lambda/\mu)$ and $\mathscr{P} \setminus \R \times (0,0,\lambda/\mu)$ are at distance 0. This is due to the presence of the set $\mathscr{S}_0$. 

We can strengthen the result in Theorem \ref{thm bif}, by showing that on all branches of non constant solutions, $\beta$ is bounded away from zero. More precisely, we have

\begin{proposition}\label{prop beta pos}
There exists $\bar \beta > 0$ such that the set of solution of \eqref{eqn model red} consists only of constant solutions if $\beta \in [0, \bar \beta)$. 
\end{proposition}
\begin{proof}
Let us assume that there exists a sequence $(\beta_n, w_{1,n}, w_{2,n}, u_{n})$ of non constant solutions of \eqref{eqn model red} such that $\beta_n > 0$ and $\beta_n \to 0$ as $n \to +\infty$. Up to striking out a subsequence, $\mathbf{v}_n :=(w_{1,n}, w_{2,n}, u_{n})$ converges to a constant solution $\mathbf{\bar v} := (\bar w_{1}, \bar w_{2}, \bar u)$ with $\beta = 0$. We have already classified these solutions in Lemma \ref{lem trivial branch of sol}. By the results in Lemmas \ref{lem 000 isolated} and \ref{lem 001 isolated}, we know that the sequence $(w_{1,n}, w_{2,n}, u_{n})$ must converge to a solution in the linear space of solution of Lemma \ref{lem trivial branch of sol}, that is the segment
\[
  s \in [0,1] \mapsto \mathbf{v}_s = \left(\frac{\lambda k - \mu \omega}{k^2} s, \frac{\lambda k - \mu \omega}{k^2} (1-s),\frac{\omega}{k}\right).
\]
We now prove that they must converge to the solution $\mathbf{v}_{1/2}$. Indeed, assume that there exists $s* \in [0,1] \setminus \{1/2\}$ such that $\mathbf{v}_n \to \mathbf{v}_{s^*}$. Without loss of generality, we can assume that $s^* > 1/2$. Thus, for $n$ sufficiently large, we have $w_{1,n} > w_{2,n}$ in $\Omega$. We define
\[
	g_n = -\beta_n w_{1,n} w_{2,n} < 0 \qquad \text{in $\Omega$}.
\]
For $n$ large, we deduce from \eqref{eqn model red} that $w_{1,n}$ and $w_{2,n}$ are both distinct solutions of the linear equation
\[
	\begin{cases}
		-\Delta w_{i,n} + (\omega - k u_n) w_{i,n} = g_n &\text{in $\Omega$}\\
		\partial_\nu w_{i,n} = 0 &\text{on $\partial \Omega$}.
	\end{cases}
\]
But then, by Fredholm's alternative it must be that the difference of any two distinct solutions is orthogonal to the zero order term $g_n$, that is
\[
	0 = \int_{\Omega} g_n (w_{1,n} - w_{2,n}) = -\beta_n \int_{\Omega} w_{1,n} w_{2,n} (w_{1,n} - w_{2,n}) < 0
\]
an obvious contradiction. 

As a result, we have $\mathbf{v}_n \to \mathbf{v}_{1/2}$ in $\mathcal{F}(\Omega)$. Thus, up to a subsequence, we can write, 
\[
 \mathbf{v}_n = \left(\frac{\lambda k - \mu \omega}{\mu \beta_n + 2k^2}, \frac{\lambda k - \mu \omega}{\mu \beta_n + 2k^2},\frac{\lambda \beta_n + 2 k \omega}{\mu \beta_n + 2 k^2} \right) + \boldsymbol{\varphi}_n
\]
where $\boldsymbol{\varphi}_n = (\varphi_{1,n}, \varphi_{2,n}, \varphi_n)$ is such that $\boldsymbol{\varphi}_n \to 0$ in $\mathcal{F}(\Omega)$ as $\beta_n \to 0$. We let
\[
  \boldsymbol{\varphi}_n = \overline{\boldsymbol{\varphi}}_n + \hat{\boldsymbol{\varphi}}_n \qquad \text{where} \quad  \overline{\boldsymbol{\varphi}}_n = \frac{1}{|\Omega|} \int_\Omega \boldsymbol{\varphi}_n
\]
Plugging these relations in system \eqref{eqn model red}, we find
\[
  \begin{cases}
    -\Delta \hat{\boldsymbol{\varphi}}_n = A_{\beta_n} \hat{\boldsymbol{\varphi}}_n + A_\beta \overline{\boldsymbol{\varphi}}_n + H(\beta_n,\boldsymbol{\varphi}_n) &\text{in $\Omega$}\\
    \partial \hat{\boldsymbol{\varphi}}_n = 0 &\text{on $\partial \Omega$}
  \end{cases}
\]
where we used notations similar to those in the proof of Theorem \ref{thm bif}. In particular, we have $\|H(\beta_n,\boldsymbol{\varphi}_n)\|_{\mathcal{F}(\Omega)} \leq C \|\boldsymbol{\varphi}_n \|_{\mathcal{F}(\Omega)}^2$ for a positive constant $C$. Let us now test the equation against $ \hat{\boldsymbol{\varphi}}_n$. We find
\begin{equation}\label{eqn spectral bif} 
  \int_\Omega |\nabla  \hat{\boldsymbol{\varphi}}_n|^2 - \int_\Omega \left<A_{\beta_n}  \hat{\boldsymbol{\varphi}}_n,  \hat{\boldsymbol{\varphi}}_n \right> = \int_\Omega H(\beta_n,\boldsymbol{\varphi}_n)  \hat{\boldsymbol{\varphi}}_n
\end{equation}
We now derive estimates for the two sides of the previous equation. First, recalling the definition of the function $H$ in \eqref{eqn rem bif}, we see that the right hand side of \eqref{eqn spectral bif}  is bounded from above by
\[
  \left|\int_\Omega H(\beta_n,\boldsymbol{\varphi}_n)  \hat{\boldsymbol{\varphi}}_n \right| \leq C \|\boldsymbol{\varphi}_n\|_{\mathcal{F}(\Omega)} \int_{\Omega} |\hat{\boldsymbol{\varphi}}_n|^2
\]
for some positive constant $C > 0$. On the other hand, by the results in Lemma \ref{lem matrix A} and \ref{lem eigen A on S} we find
\[
\begin{split}
   \int_\Omega \left<A_{\beta_n}  \hat{\boldsymbol{\varphi}}_n,  \hat{\boldsymbol{\varphi}}_n \right> =&  \beta_n\frac{\lambda k - \mu \omega}{\mu \beta_n + 2 k^2} \int_\Omega \left| (1,-1,0) \cdot  \hat{\boldsymbol{\varphi}}_n \right|^2 \\
   &+ \Re(\gamma_{1,n}) \int_{\Omega} \left| e_{1,n} \cdot  \hat{\boldsymbol{\varphi}}_n \right|^2 + \Re(\gamma_{2,n}) \int_{\Omega} \left| e_{2,n} \cdot  \hat{\boldsymbol{\varphi}}_n \right|^2
\end{split}
\]
where $ \gamma_{1,n}$ and $ \gamma_{1,n}$ are two complex conjugate eigenvalues of $A_{\beta_n}$ with strictly negative real part and $e_{1,n}$ and $e_{1,n}$ are the corresponding eigenvectors. More precisely the following holds
\[
 \beta_n\frac{\lambda k - \mu \omega}{\mu \beta_n + 2 k^2} = o_n(1) \quad \text{and} \quad \Re(\gamma_{i,n}) = -\frac{\mu \omega}{2k} + o_n(1) < 0.
\]
Since by construction the $\hat{\boldsymbol{\varphi}}_n$ have zero average on $\Omega$, by Poincar\'e's inequality, we find the following lower bound for the left hand side of \eqref{eqn spectral bif} 
\[
   \int_\Omega |\nabla  \hat{\boldsymbol{\varphi}}_n|^2 - \int_\Omega \left<A_{\beta_n}  \hat{\boldsymbol{\varphi}}_n,  \hat{\boldsymbol{\varphi}}_n \right> \geq C \int_{\Omega} |\hat{\boldsymbol{\varphi}}_n|^2.
\]
Here, the constant $C > 0$ can be chosen independent on $n$ for $n$ large enough. 

Combining the two bounds, we infer from \eqref{eqn spectral bif} that
\[
   \int_{\Omega} |\hat{\boldsymbol{\varphi}}_n|^2 \leq C \|\boldsymbol{\varphi}_n\|_{\mathcal{F}(\Omega)} \int_{\Omega} |\hat{\boldsymbol{\varphi}}_n|^2
\]
for $C >0$, a contradiction for $n$ sufficiently large.
\end{proof}

Thus we also conclude the proof of Theorem  \ref{thm bif}.

We can go further and make the conclusion of Theorem \ref{thm bif} more precise by noticing that we are in a position to apply the analytic bifurcation theory developed by Dancer in \cite{Dancer_71,Dancer_73} (see also \cite[Theorem 9.1.1]{BuffoniToland}). This approach yields the following.
\begin{theorem}\label{thm bif an}
Under the assumptions of Theorem \ref{thm bif}, for any continua of solutions $\mathscr{C}_n$ there exists a curve $\mathfrak{C}_n := \{(\beta(s), \mathbf{v}(s)) : \R \mapsto \R^+ \times \mathcal{F}(\Omega) \} \subset \mathscr{C}_n$, which contains the bifurcation point from which $\mathscr{C}_n$ emanate, such that
\begin{itemize}
	\item at any point, the curve $\mathfrak{C}_n$ can be locally reparametrized as an analytic curve;
	\item the set of possible secondary bifurcation points on $\mathfrak{C}_n$ has no accumulation points.
\end{itemize}
Moreover
\begin{itemize}
	\item either $\mathfrak{C}_n$ is a closed loop, and meets the set $\mathscr{S}_0$ in two distinct bifurcation points;
	\item or the set $\mathfrak{C}_n$ is unbounded in $(\bar \beta, +\infty) \times \mathcal{F}(\Omega)$, and more specifically
	\[
		\beta(s) \to +\infty \qquad\text{ as $s \to \infty$}.
	\]
\end{itemize}
\end{theorem}

\section{A priori estimates and strong competition singular limits}\label{sec unif est}

To shed more light on the solutions, we now derive new regularity estimates on the solutions. We are chiefly interested in estimates that are uniform in the competition parameter $\beta \geq 0$. Here we consider non negative solutions, that is solutions $\mathbf{v} = (w_1, \dots, w_N, u)$ with $w_i \geq 0$ for all $i$ and $u \geq 0$. Let us recall that, if $\beta$ is bounded, the regularity result of Lemma \ref{lem reg ell} applies also two the system of $N$ ($+1$) components. 

\begin{proposition}\label{prp asymptotic k}
Let $\Omega \subset \R^N$ be a smooth domain, and let $\beta$, $D$, $d_i$, $\omega_i$, $k_i$, $a_{ij} = a_{ji}$ for $1 \leq i,j \leq k$ be positive parameters. We consider a non negative solution $\mathbf{v} = (w_1, \dots, w_N, u) \in \mathcal{F}(\Omega)$ of the system
\begin{equation}\label{eqn model k}
	\begin{cases}
		- d_i \Delta w_i  = \left(- \omega_i + k_i u - \beta \sum_{j \neq i} a_{ij} w_j\right) w_i \\
		- D \Delta u = \left(\lambda - \mu u - \sum_{i=1}^N k_i w_i \right)u\\
		\partial_\nu w_i = \partial_\nu u = 0 &\text{ on $\partial \Omega$}.
	\end{cases}
\end{equation}
Then  all components of $\mathbf{v}$ are uniformly bounded in $L^\infty(\Omega)$ with respect to $\beta > 0$ and moreover there exists $C$ (independent of $\beta$ and $N$) such that
\[
	\| (w_1, \dots, w_N) \|_{Lip(\overline \Omega)} + \| u \|_{\C^{2,\alpha}(\overline \Omega)} \leq C
\]
If $\{\mathbf{v}_\beta\}_\beta$ is a family of non negative solutions as above, defined for $\beta \to +\infty$, then, up to subsequences, there exists $\mathbf{v} =  (w_1, \dots, w_N, u)$ with $(w_1, \dots, w_N) \in Lip(\overline{\Omega})$ and $u \in \C^{2,\alpha}(\overline \Omega)$ and
\begin{equation}\label{eqn limit segr N}
	(w_{1,\beta}, \dots, w_{N,\beta}) \to (w_1, \dots, w_N) \text{ in $\C^{0,\alpha}\cap H^1(\overline \Omega)$}, u_\beta \to u \text{ in $\C^{2,\alpha}(\overline \Omega)$}
\end{equation}
for any $\alpha \in (0,1)$. Any such limit satisfies in the sense of measures the following system of inequalities 
\begin{equation}\label{eqn segr model}
	\begin{dcases}
		- d_i \Delta w_i  \leq (- \omega_i + k_i u -\mu_i) w_i \\
		- \Delta \left( d_i w_i-\sum_{j \neq i} d_j w_j \right) \geq (- \omega_i + k_i u )w_i-\sum_{j \neq i} (- \omega_j + k_j u ) w_j &\text{in $\Omega$}\\
		- D \Delta u = \left(\lambda - \mu u - \sum_{i=1}^{k} k_i w_i \right)u\\
		\partial_\nu w_i = \partial_\nu u = 0 &\text{on $\partial \Omega$}.
	\end{dcases}
\end{equation}
Lastly, if the limit has two or more non zero components of $(w_1, \dots, w_N)$, then the subset $\{x \in \Omega: \sum_{i=1}^N w_i = 0\}$ is a rectifiable set of co-dimension 1, made of the union of a finite number of $\C^{1,\alpha}$ smooth sub-manifolds.
\end{proposition}

\begin{remark}
It can be shown that the limit in \eqref{eqn limit segr N} does not hold in the Lipschitz norm (the case $\alpha =1$), although the sequence is bounded in the Lipschitz norm. This follows from Hopf's Lemma applied to the limit functions on an regular portion of the free boundary $\{x \in \Omega: \sum_{i=1}^N w_i = 0\}$ in conjunction with the obvious fact that $\C^1(\Omega)$ is a closed subset of $Lip(\Omega)$.
\end{remark}

\begin{remark}
From the theorem we can deduce the following. Let $M$ be the number of non zero components of $(w_1, \dots, w_N)$. If $M = 0$, then either $u = 0$ or $u = \lambda / \mu$. If $M = 1$, by \cite[Theorem 1]{Mimura} the function $u$ and the unique non zero component of $(w_1, \dots, w_N)$ are positive constants and $\{x \in \Omega: \sum_{i=1}^N w_i = 0\} = \emptyset$.
\end{remark}

We shall not prove the result in all of its details, since it follows from already known ones. See for instance \cite{ContiTerraciniVerzini_AdvMat_2005} adn \cite{CaffKarLin} for the uniform estimates in H\"older spaces, \cite{SoaveZilio_ARMA} for the uniform estimate in the optimal Lipschitz norm,  \cite{CTV_indiana}, \cite{CaffKarLin} and \cite{TaTe} for the study of a closely related free-boundary problem. In any case, a complete proof for a much strong result can be found in \cite{BerestyckiZilio_NN}.

We observe that, in order to simplify the exposition, we have assumed $a_{ij} = 1$ for all $1 \leq i,j \leq k$: the results that follow can be generalized without difficulty. A much harder case arises when the competition matrix is not symmetric, that is when $a_{ij} \neq a_{ji}$ for some $i \neq j$: even though most of the results are also valid in this case, we will not consider it here, since we can only obtain a less complete description of the solutions. We refer the reader to \cite{spirals} to understand the new difficulties of this case.

We first use the uniform estimates to study more closely those bifurcation branches $\mathscr{C}_n$ which are unbounded in $\R^+ \times \mathcal{F}(\Omega)$. More generally, we will look here at any sequence of solutions $(\beta_n, w_{1,n}, w_{2,n}, u_{n})$ such that $\beta_n \to +\infty$. We recall that assumptions (H) holds, in particular $\lambda k > \mu \omega$. In the analysis of the singular limit, we use a blow-up technique first introduced in \cite{DancerDu} to study a similar situation.

We start with a key property.
\begin{lemma}\label{lem comparable}
There exists $M > 0$ such that
\[
	\frac{1}{M} \|w_{2,\beta}\|_{L^{\infty}(\Omega)}  \leq \|w_{1,\beta}\|_{L^{\infty}(\Omega)} \leq M  \|w_{2,\beta}\|_{L^{\infty}(\Omega)}
\]
for all $\mathbf{v}_\beta \in \mathscr{P}$ and $\beta$ sufficiently large.
\end{lemma}
\begin{proof}
We argue by contradiction, assuming that there exists a sequence of solutions in $\mathscr{P} $ that invalidates the conclusion. Without loss of generality, let us assume that $\|w_{1,n}\|_{L^{\infty}(\Omega)} \leq \|w_{2,n}\|_{L^{\infty}(\Omega)}$ and that the ratio $\|w_{1,n}\|_{L^{\infty}(\Omega)} / \|w_{2,n}\|_{L^{\infty}(\Omega)} \to 0$ as $\beta_n \to +\infty$. We introduce the renormalized functions
\[
	\bar w_{i,n} = \frac{w_{i,n}}{\|w_{i,n}\|_{L^{\infty}(\Omega)}} \qquad \text{for $i = 1, 2$}
\]
which are solutions to 
\[
	\begin{cases}
		- \Delta \bar w_{1,n} = - \omega \bar w_{1,n} + k \bar w_{1,n} u - \beta_n \|w_{2,n}\|_{L^{\infty}(\Omega)} \bar w_{1,n} \bar w_{2,n} &\text{ in $\Omega$}\\
		- \Delta \bar w_{2,n} = - \omega \bar w_{2,n} + k \bar w_{2,n} u - \beta_n \|w_{1,n}\|_{L^{\infty}(\Omega)} \bar w_{1,n} \bar w_{2,n} &\text{ in $\Omega$}\\
		- \Delta u_n = \lambda u_n - \mu u_n^2 - k (w_{1,n}+w_{2,n})u_n &\text{ in $\Omega$}\\
		\partial_\nu \bar w_{i,n} = \partial_\nu u = 0 &\text{ on $\partial \Omega$}
	\end{cases}
\]
We distinguish between two different cases.

\textbf{1)} $\beta_n \|w_{2,n}\|_{L^{\infty}(\Omega)}$ is bounded. In this case, all the terms in the equations are bounded uniformly with respect to $n$, and thus it is easy to see that the sequence $\bar w_{i,n}$, $u_n$ and also $w_{i,n}$, are uniformly bounded in $W^{2,p}(\Omega)$ for any $p < \infty$. Up to striking out a subsequence, we derive the strong convergence of the renormalized densities to some limit profile $(\bar w_{1,\infty}, \bar w_{2,\infty}, u_{\infty})$ with both $\bar w_{1,\infty}$ and $\bar w_{2,\infty}$ positive, while by assumption $w_{i,n} \to 0$ uniformly in $\Omega$. Moreover, by assumption we have that
\[
	 \beta_n \|w_{2,n}\|_{L^{\infty}(\Omega)} \to C \geq 0 \quad \text{while } \beta_n \|w_{1,n}\|_{L^{\infty}(\Omega)} \to 0.
\]
As a result, the limit profiles solve
\[
	\begin{cases}
		- \Delta \bar w_{1,\infty} = - \omega \bar w_{1,\infty} + k \bar w_{1,\infty} u -C \bar w_{1,\infty} \bar w_{2,\infty} &\text{ in $\Omega$}\\
		- \Delta \bar w_{2,\infty} = - \omega \bar w_{2,\infty} + k \bar w_{2,\infty} u  &\text{ in $\Omega$}\\
		- \Delta u_\infty = \lambda u_\infty - \mu u_\infty^2  &\text{ in $\Omega$}\\
		\partial_\nu \bar w_{i,\infty} = \partial_\nu u = 0 &\text{ on $\partial \Omega$}.
	\end{cases}
\]
By the maximum principle (see also Lemma \ref{lem max neumann}), we know that the equation for $u_\infty$ has only the constant solutions $u_{\infty} = 0$ or $\lambda/ \mu$. Inserting this information in the equation satisfied by $\bar w_{2,\infty}$ we see that
\[
	\begin{cases}
		- \Delta \bar w_{2,\infty} = - \omega \bar w_{2,\infty} + k \bar w_{2,\infty} u_\infty = C' w_{2,\infty}  &\text{ in $\Omega$}\\
		\partial_\nu \bar w_{2,\infty} = 0 &\text{ on $\partial \Omega$}.
	\end{cases}
\]
where the constant $C'$ is non zero by the assumption (H). It follows that necessarily $\bar w_{2,\infty} \equiv 0$, in contradiction with $\|\bar w_{2,\infty}\|_{L^{\infty}(\Omega)} = 1$.

\textbf{2)} $\beta_n \|w_{2,n}\|_{L^{\infty}(\Omega)} \to +\infty$, along a subsequence. We test the equation in $\bar w_{i,n}$ by $\bar w_{i,n}$ itself. Recalling that $\bar w_{i,n} \geq 0$ and that $u_n \leq \lambda / \mu$ we have
\[
	\int_{\Omega} | \nabla \bar w_{i,n}|^2 \leq k \frac{\lambda}{\mu} |\Omega|,
\]
where $|\Omega|$ is the measure of the set $\Omega$. Consequently, the $\bar w_{i,n}$ are bounded uniformly in $H^1(\Omega)$ and thus converge weakly to some limit $\bar w_{i,\infty} \in H^1(\Omega)$. Moreover, the compact embedding of $H^1(\Omega)$ in $L^2(\Omega)$ yields $\bar w_{i,n} \to \bar w_{i,\infty}$ strongly in $L^2(\Omega)$ and furthermore, since by construction $\|w_{i,n}\|_{L^{\infty}(\Omega)} = 1$, we have that $\bar w_{i,n} \to \bar w_{i,\infty}$ strongly in $L^p(\Omega)$ for any $p \geq 2$. Recalling that the equation for $u_n$ contains only uniformly bounded terms, up to a subsequence we have $u_n \to u_\infty$ in $W^{2,p}(\Omega)$ for any $p < \infty$. Let us show that each component of the limit configuration $(\bar w_{1,\infty}, \bar w_{2,\infty}, u_{\infty})$ is non zero. From the equations satisfied by $\bar w_{i,n}$ we know that
\[
	\begin{cases}
		- \Delta \bar w_{i,n} + \omega \bar w_{i,n} \leq k \bar w_{i,n} u_n &\text{ in $\Omega$}\\
		\partial_\nu \bar w_{i,n} = 0 &\text{ on $\partial \Omega$}.
	\end{cases}
\]
Now, letting $g_{i,n} \in H^1(\Omega)$ be the solution of
\[
	\begin{cases}
		- \Delta g_{i,n} + \omega g_{i,n} = k \bar w_{i,n} u_n &\text{ in $\Omega$}\\
		\partial_\nu g_{i,n} = 0 &\text{ on $\partial \Omega$}
	\end{cases}
\]
we have, from the previous discussion, that the sequence $\{ g_{i,n} \}_n$ is compact in $W^{2,p}(\Omega)$ for any $p > 1$ and, in particular, in $\C^{0,\alpha}(\Omega)$ for some $\alpha > 0$. On the other hand, the maximum principle yields $0 \leq \bar w_{i,n} \leq g_{i,n}$. We then assume, by way of contradiction, that either $\bar w_{i,\infty} = 0$ or $u_{\infty} = 0$. Then it follows that $g_{i,n} \to 0$ uniformly, that is, $\bar w_{i,n} \to 0$ uniformly. This is in contradiction with $\|w_{i,n}\|_{L^{\infty}(\Omega)} = 1$. 

Testing the equations in $\bar w_{1,n}$ by $\varphi \in H^1(\Omega)$, we get
\begin{equation}\label{eqn upper renorm}
	 \beta_n \|w_{2,n}\|_{L^{\infty}(\Omega)}  \int_{\Omega}\bar w_{1,n} \bar w_{2,n} \varphi = \int_{\Omega} ( k  u _n - \omega) \bar w_{1,n} \varphi - \int_{\Omega} \nabla w_{1,n} \cdot  \nabla \varphi \leq C
\end{equation}
so that, using our assumption
\[
	 \beta_n \|w_{1,n}\|_{L^{\infty}(\Omega)}  \int_{\Omega}\bar w_{1,n} \bar w_{2,n} \varphi =  \frac{\|w_{1,n}\|_{L^{\infty}(\Omega)} }{\|w_{2,n}\|_{L^{\infty}(\Omega)} } \cdot \beta_n \|w_{2,n}\|_{L^{\infty}(\Omega)}  \int_{\Omega}\bar w_{1,n} \bar w_{2,n} \varphi \to 0.
\]
As a result, $\bar w_{2,\infty}$ is a weak solution of the equation
\[
	\begin{cases}
		- \Delta \bar w_{2,\infty} = - \omega \bar w_{2,\infty} + k \bar w_{2,\infty} u_\infty  &\text{ in $\Omega$}\\
		\partial_\nu \bar w_{2,\infty} = 0 &\text{ on $\partial \Omega$}
	\end{cases}
\]
where $0 \leq u_\infty \leq \lambda / \mu$. By the maximum it follows that either $\bar w_{2,\infty} \equiv 0$ or $\bar w_{2,\infty}$ is bounded away from $0$. The former case was already excluded, thus the latter holds. But then equation \eqref{eqn upper renorm}, with $\varphi = 1$, yields
\[
	\beta_n \|w_{2,n}\|_{L^{\infty}(\Omega)} \cdot \int_{\Omega}\bar w_{1,n} \bar w_{2,n} \leq C \implies \int_{\Omega}\bar w_{1,\infty} \bar w_{2,\infty} = 0
\]
which implies $\bar w_{1,\infty} = 0$, in contradiction with the previous discussion.
\end{proof}

To push forward our analysis, in the remaining of this section we now impose a strengthening of assumption $(H)$. We require that $(\lambda k - \mu \omega)/k$ is not an eigenvalue of the Laplace operator with Neumann boundary conditions.

\begin{lemma}\label{lem asymp positive sol}
The set $\mathscr{P}$ is a pre-compact subset of $\C^{0,\alpha} \times \C^{0,\alpha} \times \C^{2,\alpha}(\overline{\Omega})$ for any $\alpha \in (0,1)$. Moreover any converging subsequence $(w_{1,n}, w_{2,n}, u_n) \to (w_{1,\infty}, w_{2, \infty}, u_\infty)$ with $\beta_n \to +\infty$ is such that 
\begin{itemize}
	\item either $(w_{1,\infty}, w_{2, \infty}, u_\infty)$ has all non zero components and, letting $V = w_{1,\infty} - w_{2,\infty}$, $V$ changes sign and $(V,u_\infty) \in \C^{2,\alpha}(\overline{\Omega})$ is a non-trivial solution of
	\[
	\begin{cases}
		- \Delta V = - \omega V + k V u &\text{ in $\Omega$}\\
		- \Delta u = \lambda u - \mu u^2 - k |V|u &\text{ in $\Omega$}\\
		\partial_\nu V = \partial_\nu u = 0 &\text{ on $\partial \Omega$,}
	\end{cases}
	\]
	\item or
	\[
		(\beta_n w_{1,n},\beta_n w_{2,n}, u_n) \sim \left(\frac{\lambda k - \mu \omega}{\mu}, \frac{\lambda k - \mu \omega}{\mu}, \frac{\lambda}{\mu}\right)
	\]
	as $n \to +\infty$ in $L^p(\Omega)$ for any $p < \infty$ and weakly in $H^1(\Omega)$.
\end{itemize}
\end{lemma}
\begin{proof}
The compactness in strong topology of the sequence of solutions was already established in Proposition \ref{prp asymptotic k}. We are left with the study of the asymptotic profiles. First of all we exclude the case $u_n \to 0$ (which would hold uniformly in $\Omega$ by the compactness properties). Indeed, in this situation we would have
\[
	\begin{cases}
		- \Delta w_{i,n} = - \omega w_{i,n} + k w_{i,n} u_n - \beta_n w_{i,n} w_{j,n} \leq - \frac{\omega}{2} w_{i,n} &\text{ in $\Omega$}\\
		\partial_\nu w_{i,n} = 0 &\text{ on $\partial \Omega$}
	\end{cases}	
\]
for $n$ sufficiently large, which implies that necessarily $w_{i,n} \equiv 0$ for $n$ large, in contradiction with the assumptions.

Let us now assume that
\[
	w_{1,n}, w_{2,n} \to 0 \qquad \text{uniformly in $\Omega$}.
\]
Passing to the limit in the equation in $u_n$, we see that $u_\infty$ satisfies
\[
	\begin{cases}
		- \Delta u_\infty = \lambda u_\infty - \mu u_\infty^2 &\text{ in $\Omega$}\\
		\partial_\nu u_\infty = 0 &\text{ on $\partial \Omega$}
	\end{cases}
\]
which implies that $u_n \to \lambda / \mu$ in $\C^{2,\alpha}(\Omega)$ (recall that we have already excluded the case $u_n \to 0$). We introduce the renormalized functions
\[
	\bar w_{i,n} := \frac{w_{i,n}}{\| w_{1,n}\|_{L^\infty(\Omega)}}
\]
which are solutions to
\[
	\begin{cases}
		- \Delta \bar w_{1,n} = - \omega \bar w_{1,n} + k \bar w_{1,n} u_n - \beta_n \|w_{1,n}\|_{L^{\infty}(\Omega)} \bar w_{1,n} \bar w_{2,n} &\text{ in $\Omega$}\\
		- \Delta \bar w_{2,n} = - \omega \bar w_{2,n} + k \bar w_{2,n} u_n - \beta_n \|w_{1,n}\|_{L^{\infty}(\Omega)} \bar w_{1,n} \bar w_{2,n} &\text{ in $\Omega$}\\
		\partial_\nu \bar w_{i,n} = 0 &\text{ on $\partial \Omega$.}
	\end{cases}
\]
Let us observe that, thanks to Lemma \ref{lem comparable}, we know that the $\bar w_{i,n}$ are bounded by some positive constant $M> 0$. By redefining $\beta_n$ as $\beta_n \|w_{1,n}\|_{L^{\infty}(\Omega)}$ and exploiting the uniform estimates of Proposition \ref{prp asymptotic k}, we have that the sequence $(\bar w_{1,n}, \bar w_{2,n})$ are precompact in $C^{0,\alpha}(\overline{\Omega})$. Moreover, using the same initial steps as in Case 2) of Lemma \ref{lem comparable}, we see that $\bar w_{i,n} \to \bar w_{i,\infty}$ in $L^p(\Omega)$ for any $p < \infty$ and weakly in $H^1(\Omega)$ and also $\bar w_{i,\infty} \neq 0$. Letting $V_n = \bar w_{1,n} - \bar w_{2,n}$, we see that $\|V_n\|_{L^\infty(\Omega)} \leq M+1$ and
\[
	\begin{cases}
		- \Delta V_n = (-\omega + k u_n) V_n &\text{ in $\Omega$}\\
		\partial_\nu V_n = 0 &\text{ on $\partial \Omega$}.
	\end{cases}
\]
As a consequence of the strong convergence $u_n \to \lambda / \mu$, we see that $V_n \to V_\infty$ in $\C^{2,\alpha}(\Omega)$. This function $V_\infty$ is a solution of the following limit equation
\[
	\begin{cases}
		- \Delta V_\infty = \frac{\lambda k - \mu \omega}{\mu} V_\infty &\text{ in $\Omega$}\\
		\partial_\nu V_\infty = 0 &\text{ on $\partial \Omega$}
	\end{cases}
\]
where, by assumption, $(\lambda k - \mu \omega) / \mu \neq \gamma_i$, the eigenvalues of the Laplacian with Neumann boundary conditions. Consequently, $V_\infty \equiv 0$ and thus $\bar w_{1,\infty} = \bar w_{2,\infty} \neq 0$. Testing the equation in $\bar w_{i,n}$ by $\bar w_{i,n}$ itself, we find
\[
	\int_{\Omega} |\nabla \bar w_{i,n}|^2 + \beta_n \|w_{1,n}\|_{L^{\infty}(\Omega)} \int_{\Omega} \bar w_{i,n}^2 \bar w_{j,n}  = \int_{\Omega} \left(- \omega + k u_n\right) \bar w_{i,n}^2  \leq C
\]
which implies, in particular, that $\beta_n \|w_{1,n}\|_{L^{\infty}(\Omega)} \to C$ for some constant $C \geq 0$. Finally, passing to the limit in the equation in $\bar w_{i,n}$ we find
\[
	\begin{cases}
		- \Delta \bar w_{\infty} =  \left( - \omega  + k \frac{\lambda}{\mu}\right) \bar w_{\infty} - C \bar w_{\infty}^2 &\text{ in $\Omega$}\\
		\partial_\nu \bar w_{\infty} = 0 &\text{ on $\partial \Omega$}.
	\end{cases}	
\]
If $C = 0$, since $(\lambda k - \mu \omega) / \mu > 0$ and $\bar w_{\infty}$ is non negative by the maximum principle, it must be the case that $\bar w_{\infty} \equiv 0$, in contradiction with the renormalization. Thus $C > 0$ and, from a direct application of the maximum principle (see Lemma \ref{lem max neumann}), we have that the only non negative solution to the previous equation are the constant. In particular, it must be the case that $\bar w_{\infty} \equiv 1$, thus
\[
	C = \frac{\lambda k - \mu \omega}{\mu}. \qedhere
\]
\end{proof}

\section{One dimensional case}\label{sec dim one}

In the one-dimensional case the description of the bifurcation branches is more complete.
\begin{theorem}\label{thm bif one}
Under the assumptions of Theorem \ref{thm bif}, let us moreover suppose that $\Omega \subset \R$ is an open and bounded interval. Then any eigenvalue $\gamma$ of \eqref{eqn eigen lap neu} is of multiplicity one, and the corresponding continuum of solutions $\mathscr{C}_n$ (and $\mathfrak{C}_n$) generated from the set $\mathscr{S}_0$ at the value
\[
	\beta_n\frac{\lambda k - \mu \omega}{\mu \beta_n + 2 k^2} = \gamma_n
\]
is unbounded and it intersects the set $\mathscr{S}_0$ only once.
\end{theorem}
\begin{proof}
The proof follows again the main ideas presented in \cite{Rabinowitz_JFA}. Under the assumptions, there exist $a < b \in \R$ such that $\Omega = (a,b) \subset \R$. We can explicitly compute the eigenvalues of \eqref{eqn eigen lap neu}, which are given by 
\[
	\gamma_n := \left(\frac{\pi}{b-a} n\right)^2 \qquad \text{for any $n \in \N$}.
\]
Based on the discussion in Theorem \ref{thm bif}, any value of $\gamma_n$ corresponds to a bifurcation point. This is true even for the set $\mathscr{C}_0$ generating from $\gamma_0 = 0$, which is given by a trivial linear subspace of constant solutions, subject of Lemma \ref{lem trivial branch of sol}.

Let us consider, for a fixed $\gamma_n$ with $n \geq 1$ as before, the continuum of solutions $\mathscr{C}_n$ that emerges from the set $\mathscr{S}_0$. By the perturbations analysis conducted in Theorem \ref{thm bif}, we know that the solutions are of the form
\[%\begin{multline*}
	(w_{1,\beta}, w_{2,\beta}, u_{\beta}) = \left(\frac{\lambda k - \mu \omega}{\mu \beta_n + 2 k^2}, \frac{\lambda k - \mu \omega}{\mu \beta_n + 2 k^2}, \frac{\lambda \beta_n + 2 k \omega}{\mu \beta_n + 2 k^2} \right) + \eps (\psi_n, -\psi_n,0) + o(\eps)	
\]%\end{multline*}
where $\eps$ is a parameter such that $\eps \to 0$ when $\beta \to \beta_n$, $\psi_n$ is a normalized eigenfunction of \ref{eqn eigen lap neu} in $\Omega = (a,b)$ and $o(\eps)$ is a perturbation in $\C^{2,\alpha}([a,b])$ of order less than $\eps$. In particular, letting
\[
	v_{\beta,n} = w_{1,\beta} - w_{2,\beta} = 2 \eps \psi_n + o(\eps)
\]
(where we have explicitly stated the index $n$ of the eigenfunction which spans $v_{\beta,n}$) we have that $v_{\beta,n}$ solves
\begin{equation}\label{eqn ode}
	\begin{cases}
		- v_{\beta,n}'' = (- \omega + k u_\beta ) v_{\beta,n} &\text{ in $(a,b)$,}\\
		v_{\beta,n}'(a) = v_{\beta,n}'(b) = 0.
	\end{cases}
\end{equation}
As a result, when $\eps$ is small, $v_{\beta}$ has exactly $n$ distinct simple zeroes in $(a,b)$, located closely to the zeroes of the eigenfunction $\psi_n$. We recall that the solutions of the system \eqref{eqn model red} are bounded in $Lip([a,b])$ uniformly with respect to $\beta$ and in particular the last component $u_\beta$ is bounded in $\C^{2,\alpha}([a,b])$ for all $\alpha < 1$: it follows that there exists a parametrization of the continuum $\mathscr{C}_n$ with respect to which the functions $v_{\beta,n}$ vary smoothly and they also are uniformly bounded in $\C^{2,\alpha}([a,b])$ for all $\alpha < 1$.

We claim that on each continuum of solutions $\mathscr{C}_n$, the number of zeroes of the function $v_{\beta,n}$ does not change. To this end, we adapt a classical argument for scalar equations to the present situation of a system. To prove the claim, we first observe that the solutions depend smoothly in $\C^{2,\alpha}([a,b])$ on the parametrization. Thus, if $v_{\beta,n}$ changes the number of zeros, there exists a solution $v$ of \eqref{eqn ode} inside $\mathscr{C}_n$ that has a zeros of multiplicity at least two. This point could be located at the interior or at the boundary of the interval $(a,b)$, thanks to the homogeneous Neumann condition. The uniqueness theorem for ordinary differential equations with smooth coefficients applied to \eqref{eqn ode} implies that the function $v$ must be identically 0. As a result, the corresponding solution $(w_1,w_2,u)$ has first and second components that are equal: reasoning as in Lemma \ref{lem trivial branch of sol}, by letting $V = w_1 + w_2$, we obtain a solution to
\[
	\begin{cases}
		- V'' = - \omega V + k V u - \beta V^2&\text{ in $(a,b)$}\\
		- u'' = \lambda u - \mu u^2 - k V u &\text{ in $(a,b)$}\\
		V' = u' = 0 &\text{ on $\{a,b\}$}.
	\end{cases}
\]
and again thanks to the results in \cite{Mimura}, the solution $(w_1,w_2,u)$ must be a constant solution. By Proposition \ref{prop beta pos}, we already know that $\beta > 0$. Thus $(w_1,w_2,u)$ belongs necessarily to the set $\mathscr{S}_0$ (we recall that if $\beta > 0$ the only bifurcation points belong to $\mathscr{S}_0$, see Remark \ref{rem stab indip beta}). From the previous discussion it must be that the point corresponds to a different eigenvalue $\gamma_m$, $m \neq n$, and locally the solutions can be written as a perturbation along the line spanned by the eigenfunction $\psi_m$. In the same vein of the previous argument, the difference of the first two components can be asymptotically expanded as
\[
	v_{\beta,m} = w_{1,\beta} - w_{2,\beta} = 2 \eps \psi_m + o(\eps)
\]
and, again, for $\eps \to 0$, the solution $v_{\beta,m}$ has $m$ distinct simple zeroes on $(a,b)$. In particular, the solution must have $m \neq n$ zeros in a neighborhood of the bifurcation point, leading us to a contradiction.
\end{proof}

We can show the following stronger version of Lemma \ref{lem asymp positive sol}, which completes the analysis of the bifurcation diagram in  dimension one.
\begin{lemma}\label{lem asymp positive 1d}
Let $\Omega \subset \R$ be an open and bounded interval. Then same conclusion of Lemma \ref{lem asymp positive sol}. Moreover, let $(w_{1,n}, w_{2,n}, u_n) \in \mathscr{P}$ be any converging sequence for $\beta_n \to +\infty$ and let $(w_{1,\infty}, w_{2, \infty}, u_\infty)$ be its limit. Then either $(w_{1,\infty}, w_{2, \infty}, u_\infty)$ has all its components non zero or for $n$ large it holds
\[
		(w_{1,n}, w_{2,n}, u_n) =  \left(\frac{\lambda k - \mu \omega}{\mu \beta_n + 2k^2}, \frac{\lambda k - \mu \omega}{\mu \beta_n + 2k^2},\frac{\lambda \beta_n + 2 k \omega}{\mu \beta_n + 2 k^2} \right).
\]
\end{lemma}
\begin{proof}
As in Theorem \ref{thm bif one}, we can use the auxiliary function $v_\beta = w_{1,\beta}  - w_{2,\beta}$ to study more accurately the second case of the lemma. The conclusion is reached once again by counting the number of zeros of $v_\beta$ and observing that this must be constant on each bifurcation branch.
\end{proof}
As a direct consequence, we have that there exists $\delta > 0$ such that, for $\beta'$ sufficiently large,
\[
	\mathrm{dist} \left(\mathscr{P} \setminus \mathscr{S}_0, \mathscr{S}_0 \right) > \delta \qquad \text{for all $\beta \geq \beta'$}
\]
where the distance is taken in the sense of the $\C^{0,\alpha} \times \C^{0,\alpha} \times \C^{2,\alpha}(\overline{\Omega})$ norm for any $\alpha \in (0,1)$. Moreover each branch of solutions constructed in Theorem \ref{thm bif one} converge (up to a subsequence) to a disjoint set of solutions for the limit problem, characterized by the different number of zeroes for the function $V$ and do not re-collapse on the set $\mathscr{S}_0$ for $\beta = +\infty$.

\section{Maximum number of densities}\label{sec opt rep}
We continue the investigation of the model by addressing a fundamental question: can the model be used to predict the optimal repartition of the domain $\Omega$ in hunting territories, that is, the optimal number of packs?

To answer this question, we first focus on the limit stationary system satisfied by the densities in the case of segregation. We shall prove two complementary results in this direction.
\begin{enumerate}
	\item We show that each bounded domain $\Omega \Subset \R^n$ can sustain a maximum number of densities of predators (see Lemma \ref{lem max packs} and  Theorem \ref{thm max packs}). As we will see, this implies that there exists a number $k \geq 1$ of packs that, for a given configuration of parameters, maximizes the total population of predators.
	\item Then, we show that, for a particular choices of the parameters, the total population of predators in the case of two packs is strictly higher than that of one pack only, implying that in these cases the optimal configuration is given by a finite number of packs strictly greater than two.
\end{enumerate}

We start with the following result, which states that for each environment $\Omega$ there is a maximal number of densities of predators $\bar N$ that can be sustained.
\begin{lemma}\label{lem max packs}
For a given smooth domain $\Omega \subset \R^n$, there exists $\bar N \in \N$ such that any non negative solution $(w_1, \dots, w_N, u) \in H^1(\Omega)$ of \eqref{eqn segr model} has at most $\bar N$ non zero components among $(w_1, \dots, w_N)$.
\end{lemma}
In order to prove the previous result, we need to recall the notion of optimal partition (see for instance \cite{HelfHofOstTerr_2009} for a general survey and some fundamental results). For consistency with the theory of optimal partitions, in the next two results, eigenvalues will be counted starting from the index 1. For any $1 \leq h \in \N$ we say that a family $\mathcal{D} = \{ D_1, \dots, D_h\}$ of open subsets of $\Omega$ is a $h$ (open) partition of $\Omega$ if
\[
	D_i \cap D_j = \emptyset \; \forall i \neq j \text{ and } \overline{\cup_{i=1}^h D_i} = \overline{\Omega}.
\]
For each $D_i$, we define the generalized first eigenvalue as
\[
	\gamma_1(D_i) := \inf \left\{ \left. \int_{\Omega} |\nabla \varphi|^2 \right/ \int_{\Omega} \varphi^2 : \varphi \in H^1(\Omega), \; \varphi = 0 \text{ in $\Omega \setminus D_i$} \right\}
\]
and for the partition $\mathcal{D}$ we assign the total value
\[
	\Lambda(\mathcal{D}) = \max_{i} \gamma_1(D_i).
\]
A partition $\mathcal{D}$ is optimal if it minimizes the value of $\Lambda(\mathcal{D})$ among all $N$-partitions. We recall the following result (see \cite[Corollary 5.6]{HelfHofOstTerr_2009}), which follows from the Courant-Fischer-Weyl characterization of the eigenvalues of compact hermitian operators.
\begin{theorem}\label{thm eig bound hhot}
Let $\gamma_h(\Omega)$ be the $h$-th eigenvalue (counted with multiplicity) of
\[
	\begin{cases}
		-\Delta \varphi = \gamma \varphi &\text{in $\Omega$}\\
		\partial_{\nu} \varphi = 0 &\text{on $\partial \Omega$}.
	\end{cases}
\]
Then 
\[
	\Lambda(\mathcal{D}) \geq \gamma_h(\Omega)
\]
for all $h$-partitions $\mathcal{D}$ of $\Omega$.
\end{theorem}
\begin{proof}[Proof of Lemma \ref{lem max packs}]
Consider the solution $(w_1, \dots, w_N,u)$ of \eqref{eqn segr model}. If the component $u$ is zero, all the remaining components of the solution are zero. This can be derived by testing the inequalities in \eqref{eqn segr model} with $w_i \in Lip(\overline{\Omega})$. Indeed, we thus obtain
\[
	\int_{\Omega} d_i |\nabla w_i|^2 + ( \omega_i - k_i u)  w_i^2 = 0
\]
which yields the claim taking into account that $u \equiv 0$. As a result, we can assume $u \geq 0$ and $u \not \equiv 0$. By the maximum principle applied to the equation in $u$, we find that $u$ is strictly positive. Moreover, since $w_i \geq 0$ for $i =1, \dots, k$, we have
\[
	- D \Delta u = \left(\lambda - \mu u - \sum_{i=1}^N k_i w_i \right)u \leq \left(\lambda - \mu u \right)u \quad \text{ thus } u \leq \frac{\lambda}{\mu}.
\]
On the other hand, we have
\[
	- d_i \Delta w_i  = \left(- \omega_i + k_i u\right) w_i \leq \left(- \omega_i + k_i \frac{\lambda}{\mu} \right) w_i.
\]
That is, letting $\Omega_i := \{ w_i > 0\}$, $w_i$ satisfies 
\[
	\begin{cases}
		- d_i \Delta w_i  \leq \left(- \omega_i + k_i \frac{\lambda}{\mu}\right) w_i &\text{in $\Omega_i$}\\
		w_i = 0 &\text{on $\partial \Omega_i \cap \Omega$}\\
		\partial_\nu  w_i = 0 &\text{on $\partial \Omega_i \cap \partial \Omega$}
	\end{cases}
\]
Considering the first eigenvalue $\gamma_1(\Omega_i)$ of $\Omega_i$, that is the minimal value of the following mixed problem in $\Omega_i$:
\[
	\begin{cases}
		- \Delta \varphi_i  = \gamma_1(\Omega_i) \varphi_i &\text{in $\Omega_i$}\\
		\varphi_i = 0 &\text{on $\partial \Omega_i \cap \Omega$}\\
		\partial_\nu  \varphi_i  = 0 &\text{on $\partial \Omega_i \cap \partial \Omega$}
	\end{cases}
\]
with $\varphi_i \neq 0$. It is known that $\gamma_1(\Omega_i) \geq 0$ and $\varphi_i > 0$ in $\Omega_i$. From the comparison principle, it follows that
\[
	\text{if} \quad \frac{\lambda k_i - \mu \omega_i}{d_i \mu} < \gamma_1(\Omega_i) \quad \text{then} \quad w_i \equiv 0.
\]
In particular we see that if all the components $w_1, \dots, w_N$ are $\not \equiv 0$ then necessarily 
\[
	\max_{i = 1, \dots, k} \gamma_1(\Omega_i) < \max_{i = 1, \dots, k} \frac{\lambda k_i - \mu \omega_i}{d_i \mu} = \bar \gamma.
\]
Since $\Omega_1, \dots, \Omega_N$ is a $N$-partition of the set $\Omega$, we evince by Theorem \ref{thm eig bound hhot} that necessarily
\[
	\gamma_N(\Omega) < \bar \gamma.
\]
As a result, we reach the desired conclusion recalling that the sequence of eigenvalues $\gamma_1 < \gamma_2 \leq \gamma_3 \leq \dots$ is unbounded.
\end{proof}
Using Weyl's asymptotic law for the Laplacian with Neumann boundary conditions, we can obtain a more explicit bound on the constant $\bar N$. This is stated in the next theorem.

\begin{theorem}
The number $\bar N$ of Lemma \ref{lem max packs} admits the following asymptotic upper bound
\[
	\bar N \lesssim \frac{\omega_n}{(2\pi)^n} |\Omega| \left(\max_{i = 1, \dots, k} \frac{\lambda k_i - \mu \omega_i}{d_i \mu}\right)^{n/2} \qquad \text{for $\max_{i = 1, \dots, k} \frac{\lambda k_i - \mu \omega_i}{d_i \mu} \to +\infty$}.
\]
where $\omega_n$ is the volume of the unit sphere in $\R^n$.
\end{theorem}
We recall that $A \lesssim B$ for $B \to +\infty$ means that $A \leq B + o(B)$ and $o(B)/ B \to 0$ as $B \to +\infty$. This estimate agrees with the intuition that doubling the size of the domain $\Omega$ would also double the number of groups of predators it can sustain.

\begin{proof}
For a fixed  smooth domain $\Omega \subset \R^n$, if we let $N(\gamma)$ stand for the number of eigenvalues for the Laplace operator with homogeneous Neumann boundary conditions in $\Omega$ which are less than $\gamma$, it can be shown that
\[
	N(\gamma) = \frac{\omega_n}{(2\pi)^n} |\Omega| \gamma^{n/2} + o(\gamma^{n/2}).
\]
Substituting this expression in the bound found in Lemma \ref{lem max packs} proves the statement.
\end{proof}

We can now extend the result in Lemma \ref{lem max packs} to the original competitive system. We show in particular
\begin{theorem}\label{thm max packs}
There $\bar \beta > 0$ such that if $\beta > \bar \beta$ and $\mathbf{v}_\beta = (w_{1,\beta}, \dots, w_{N,\beta}, u_{\beta})$ is a solution to \eqref{eqn model} then
\begin{itemize}
	\item either at most $\bar N$ components of $\mathbf{w}_\beta = (w_{1,\beta}, \dots, w_{N,\beta})$ are strictly positive and the others are zero;
	\item or the solution is such that
\[
	 \|(w_{1,\beta},\dots,w_{N,\beta})\|_{\C^{0,\alpha}(\Omega)} + \|u_\beta- \lambda/\mu\|_{\C^{2,\alpha}(\Omega)} = o_\beta(1)
\]
for every $\alpha \in (0,1)$.
\end{itemize}
\end{theorem}

We recall that, by Proposition \ref{prp asymptotic k}, the set of solutions of \eqref{eqn model red} is pre-compact in $H^1(\Omega) \cap \C^{0,\alpha}(\overline{\Omega})$.

\begin{proof}
The statement will follow from some approximation results in combination with Lemma \ref{lem max packs}. First of all, we want to show that if $h > \bar N$ components of $\mathbf{w}_\beta$ are non zero and $\beta$ sufficiently large, then the solutions is close to the solutions $(0, \dots, 0, 0)$ or $(0, \dots, 0, \lambda/\mu)$. We have
\begin{claim} 
Let $\mathbf{v}_\beta = (\mathbf{w}_{\beta}, u_{\beta}) \in \mathcal{F}(\Omega)$ be a family of solutions to \eqref{eqn model k} and let us assume that there exists a solution $\mathbf{v} = (\mathbf{w}, u) = (w_1, \dots, w_N, u)$ with $(w_1, \dots, w_N) \in Lip(\overline{\Omega})$ and $u \in \C^{2,\alpha}(\overline \Omega)$ to \eqref{eqn segr model} with $h$ components of $\mathbf{w}$ non zero (with $1 \leq h \leq N$) such that $\mathbf{v}_\beta \to \mathbf{v}$ as $\beta \to +\infty$ in $H^1(\Omega) \cap \C^{0,\alpha}(\overline{\Omega})$. Then there exists $\bar \beta > 0$ sufficiently large such that the solution $\mathbf{v}_\beta$ has exactly $h$ components of $\mathbf{w}_\beta$ that are non zero for $\beta \geq \bar \beta$.
\end{claim}
Let us first show how to use the claim in order to derive the first part of the theorem. Consider a family of solutions $(\mathbf{w}_{\beta}, u_{\beta}) \in \mathcal{F}(\Omega)$ of the system \eqref{eqn model k}, with $\beta \to +\infty$. By Lemma \ref{lem max packs} we already know that any solution of \eqref{eqn segr model} has at most $\bar N$ non trivial components among those of $\mathbf{w}_\beta$. Let us assume that $\mathbf{w}_\beta$ contains a sub-family (which we shall not relabel) with $\beta \to +\infty$ that has more than $\bar N$ non zero components. Proposition \ref{prp asymptotic k} implies that, up to a subsequence, $\mathbf{v}_\beta \to \mathbf{v}$ in $H^1(\Omega) \cap \C^{0,\alpha}(\overline{\Omega})$, where $\mathbf{v}$ solves the limit system \eqref{eqn segr model}, and thus has at most $\bar N$ components among those of $\mathbf{w}$, in contradiction with our claim.

We now prove the claim, arguing by contradiction and adopting the scheme of Lemma \ref{lem comparable}. Let $\mathbf{v}_n$ be any sequence satisfying the assumptions of the claim and let $\mathbf{v}$ be its limit for $\beta_n \to +\infty$. By the maximum principle, if the corresponding $\mathbf{w}$ has $h$ non trivial components, then necessarily $u$ is strictly positive. Up to a relabelling, we can assume that the first $h$ components $(w_1, \dots, w_h)$ are also non zero, while $w_{h+1} = \dots =  w_N = 0$. As a result, the sub-vector $(w_1, \dots, w_h, u)$ satisfies the conclusions of Proposition \ref{prp asymptotic k}, and, in particular, the set 
\[
	\mathcal{N}:= \left\{x \in \Omega : \sum_{i=1}^h w_i(x) = 0 \right\}
\]
is a rectifiable set of co-dimension 1, made of the union of a finite number of $\C^{1,\alpha}$ smooth sub-manifolds (points if $\Omega \subset \R$, curves if $\Omega \subset \R^2$, and in general embedded surfaces if $\Omega \subset \R^N$). For any $n \in \N$, we introduce the renormalized solution
\[
	 \mathbf{\bar v}_n := \left(\frac{w_{1,\beta_n}}{\|w_{1,\beta_n}\|_{L^\infty(\Omega)}}, \dots, \frac{w_{N,\beta_n}}{\|w_{N,\beta_n}\|_{L^\infty(\Omega)}}, u_{\beta_n}\right)
\]
which is well defined since, by assumption, for any $n$, and $1 \leq i \leq N$, $w_{i,n} > 0$. Let us observe that, since the first $h$ components of $\mathbf{v}_n$ converge to non zero functions as $\beta \to +\infty$, the $L^\infty$ norm of these components is bound from above and away from zero uniformly in $n$. That is, for the first $h$ components of $\mathbf{\bar v}_n$ we have
\[
	\bar{w}_{i,\beta_n} = C_{i,n} w_{i,\beta_n} \qquad \text{with} \quad 0 < C < C_{i,n} < C^{-1} \quad \forall 1\leq i \leq k, n \in \N.
\]
By the convergence of $\mathbf{w}_{\beta_n}$ we deduce that the first $h$ components of $\mathbf{\bar w}_{\beta_n}$ converge in $H^1(\Omega) \cap \C^{0,\alpha}(\overline{\Omega})$ to a non zero scaling of the first $h$ components of $\mathbf{\bar w}$.  Consequently $(\bar w_{1}, \dots, \bar w_h)$ and $(w_{1}, \dots, w_h)$ share the zero set $\mathcal{N}$.

Including the scaling in the system, it follows that $\mathbf{\bar v}_n$ is a solution of 
\begin{equation}\label{eqn model k scaled}
	\begin{cases}
		- d_i \Delta \bar w_{i,\beta_n}  = \left(- \omega_i + k_i u_{\beta_n} - \beta \sum_{j \neq i} a_{ij}\|w_{j,\beta_n}\|_{L^\infty(\Omega)}  w_{j,\beta_n}\right) \bar w_{i,\beta_n} \\
		- D \Delta u_{\beta_n} = \left(\lambda - \mu u_{\beta_n} - \sum_{i=1}^N k_i w_{i,\beta_n} \right)u_{\beta_n}\\
		\partial_\nu \bar w_{i,\beta_n} = \partial_\nu u_{\beta_n} = 0 &\text{ on $\partial \Omega$}
	\end{cases}
\end{equation}
We are chiefly interested in the equations satisfied by the densities $\bar w_{i,\beta_n}$ for $h+1\leq i \leq N$. The maximum principle implies that $u_n \leq \lambda / \mu$: testing the $i$-th equation with the density $\bar w_{i,\beta_n}$ itself and using its positivity, we find
\[
	\int_{\Omega} |\nabla \bar w_{i,\beta_n}|^2 \leq \frac{k_i}{d_i} \frac{\lambda}{\mu} |\Omega|.
\]
Since by definition $\|\bar w_{1,\beta_n}\|_{L^\infty(\Omega)} = 1$, $\bar w_{i,\beta_n}$ is uniformly bounded in $H^1(\Omega)$ and it admits a weak limit $\bar w_i \in H^1(\Omega)$. The compact embedding in $L^2(\Omega)$ and the boundedness in $L^\infty(\Omega)$ also imply that $\bar w_{i,\beta_n} \to \bar w_i$ strongly in $L^p(\Omega)$ for any $p < \infty$. Let us show that $\bar w_i$ is not zero: for each $h+1 \leq i \leq k$ (the other components are non trivial by assumption) we have that
\[
	\begin{cases}
		- d_i \Delta \bar w_{i,\beta_n} \leq \left(- \omega_i + k_i u_{\beta_n} \right) \bar w_{i,\beta_n} \\
		\partial_\nu \bar w_{i,\beta_n} = 0 &\text{ on $\partial \Omega$}.
	\end{cases}
\]
Let $g_{i,n} \in H^1(\Omega)$ be a solution to 
\[
	\begin{cases}
		- d_i \Delta g_{i,n} + \omega_i g_{i,n} = k_i u_{\beta_n} \bar w_{i,\beta_n} \\
		\partial_\nu g_{i,\beta_n} = 0 &\text{ on $\partial \Omega$}.
	\end{cases}
\]
By standard arguments we know that $0 \leq \bar w_{i,\beta_n} \leq g_{i,n}$ and that
\[
	\|g_{i,n}\|_{\C^{0,\alpha}(\Omega)} \leq C \|g_{i,n}\|_{W^{2,p}(\Omega)} \leq C \frac{k_i}{d_i} \frac{\lambda}{\mu} \|\bar w_{i,\beta_n} \|_{L^p(\Omega)}
\]
for any $N/2 < p < \infty$ and suitable $C$ and $\alpha > 0$. As a result, using the order relationship between $w_{i,\beta_n}$ and $g_{i,n}$, we have
\[
	1 =  \|\bar w_{i,\beta_n} \|_{L^\infty(\Omega)} \leq C \frac{k_i}{d_i} \frac{\lambda}{\mu} \|\bar w_{i,\beta_n} \|_{L^p(\Omega)}.
\]
Thus, the strong limit $\mathbf{\bar v}$ in $L^p(\Omega)$ of $\mathbf{\bar v}_n$ has all of its components which are non trivial. We are now in a position to reach the aimed contradiction. Let us consider the equation satisfied by $\bar w_{i,\beta}$ for $h+1 \leq i \leq k$: for simplicity, we scale back the first $h$ densities and find that
\[
	\begin{cases}
		- d_i \Delta \bar w_{i,\beta_n} \leq \left(- \omega_i + k_i u_{\beta_n} - \beta_n \sum_{1\leq j \leq h} w_{j,\beta_n}\right) \bar w_{i,\beta_n} \\
		\partial_\nu \bar w_{i,\beta_n} = 0 &\text{ on $\partial \Omega$}.
	\end{cases}
\]
From the previous discussion, $\bar w_{i,\beta_n} \to \bar w_{i}$ in $L^2(\Omega)$ and the limit is non trivial. We let
\[
	c_i := \int_{\Omega} \bar w_{i}^2 > 0.
\]
For any $\eps > 0$ and $n \in \N$, we consider the sets
\[
	\Omega_{\eps,n} := \left\{ x \in \Omega : \dist(x,\partial \Omega) \geq \eps \text{ and } \inf_{m \geq n} \left(\sum_{i=1}^h w_{i,\beta_m}(x)\right) \geq \eps \right\}.
\]
By the uniform convergence of $\mathbf{v}_{\beta_n}$ and the properties of its limit configuration we know that each $\Omega_{\eps,n}$ is closed and $\Omega_{\eps_1,n_1} \subseteq \Omega_{\eps_2,n_2}$ whenever $\eps_1 > \eps_2$ and $n_1 < n_2$. Finally,
\[
	\bigcup_{\eps > 0, n \in \N} \Omega_{\eps,n} = \Omega \setminus \mathcal{N}.
\]
As we have already recalled, $\mathcal{L}^N(\mathcal{N}) = 0$. Hence, it follows that for any $\delta > 0$, there exist $\bar \eps > 0 $ and $n^* \in \N$ such that $\mathcal{L}^N( \Omega \triangle \Omega_{\eps,n}) \leq \delta$ for $ 0 < \eps < \bar \eps$ and $n \geq n^*$ (here $A \triangle B$ is the symmetric difference of the sets $A$ and $B$). By the absolute continuity of the Lebesgue integral and the uniform integrability of converging sequences, there exists $\bar \delta > 0$ (and consequently $\bar \eps$ and $n^*$) such that
\[
	\int_{\Omega_{\bar \eps,n^*} } \bar w_{i,\beta_m}^2 \geq \frac{c_i}{2} > 0 \qquad \text{for $m \in \N$ sufficiently large}.
\]
On the other hand, testing the equation in $\bar w_{i,\beta_m}$ by $\bar w_{i,\beta_m}$ itself, we obtain
\[
	\int_{\Omega} \left[ d_i |\nabla \bar w_{i,\beta_m}|^2 + \omega_i \bar w_{i,\beta_m}^2 + \beta_m \left(\sum_{1\leq j \leq h} w_{j,\beta_n}\right) w_{i,\beta_m}^2 \right]
	\leq k_i \int_\Omega u_{\beta_m} w_{i,\beta_m}^2 \leq k_i \frac{\lambda}{\mu} |\Omega|.
\]
Since the terms of the left hand side are positive, we can localize the integral on the sets $\Omega_{\bar \eps,n^*}$ and find
\[
	\beta_m \inf_{\Omega_{\bar \eps, n^*}} \left(\sum_{1\leq j \leq h} w_{j,\beta_n}\right)  \int_{\Omega_{\bar \eps, n^*}} w_{i,\beta_m}^2 \leq k_i \frac{\lambda}{\mu} |\Omega|
\]
that is
\[
	0 < \frac{c_i}{2} < \int_{\Omega_{\bar \eps,n^*} } \bar w_{i,\beta_m}^2  \leq  \frac{1}{\beta_m \bar \eps} \cdot  k_i \frac{\lambda}{\mu} |\Omega|
\]
a contradiction when $\beta_m$ is sufficiently large, and this proves the first claim.

At this point, we have established that positive solutions must converge to either one of the solutions  $(0, \dots, 0, 0)$ or $(0, \dots, 0, \lambda/\mu)$. To conclude the proof, we show that they can only converge to the latter. For this, we can follow the same reasoning as for Lemma \ref{lem asymp positive sol}. That is, we suppose that $\mathbf{v}_n \to (0,\dots, 0,0)$ in $H^1(\Omega) \cap \C^{0,\alpha}(\overline{\Omega})$ and find that for $n$ large enough
\[
	\begin{cases}
		- d_i \Delta w_{i,\beta_n} = \left(- \omega_i + k_i u_{\beta_n} - \beta_n \sum_{ j \neq i} w_{j,\beta_n}\right) w_{i,\beta_n} \leq - \frac{\omega_i}{2} w_{i,\beta_n} \\
		\partial_\nu w_{i,\beta_n} = 0 &\text{ on $\partial \Omega$}
	\end{cases}
\]
which implies that each $w_{i,\beta_n}$ must be identically zero, in contradiction with our positivity assumption. This concludes the proof of Theorem \ref{thm max packs}.
\end{proof}

Combining the previous results, we can state the following
\begin{theorem}\label{thm maxim}
Let $\delta > 0$ and, for arbitrary $N \geq 1$, let us consider an arbitrary family of coefficients such that
\[
	\delta < d_1, \dots, d_N, \omega_1, \dots, \omega_N, k_1, \dots, k_N < \frac{1}{\delta}.
\]
For $\beta \geq 0$, let $\mathscr{S}$ be the set of solutions $(w_1, \dots, w_N, u) \in \mathcal{F}(\Omega)$ to \eqref{eqn model k} with any number of components and coefficients as above. For any $(w_1, \dots, w_N, u) \in \mathscr{S}$ we associated
\[
	P(w_1, \dots, w_N, u) = \int_{\Omega} \sum_{i=1}^N w_i.
\]
Then $\bar N \in \N$ for which we have two alternatives
\begin{itemize} 
	\item either there exists $(\bar w_1, \dots, \bar w_{\bar N}, \bar u) \in \mathscr{S}$ such that
	\[
		P(\bar w_1, \dots, \bar w_{\bar N}, \bar u) = \max_{N \geq 1, (w_1, \dots, w_N, u) \in \mathscr{S}} P(w_1, \dots, w_N, u);
	\]
	\item there exist a sequence $(\bar w_{1,n}, \dots, \bar w_{\bar N,n}, \bar u_n) \in \mathscr{S}$ and functions $(w_1, \dots, w_N) \in Lip(\overline{\Omega})$ and $u \in \C^{2,\alpha}(\overline \Omega)$ such that
	\begin{itemize}
		\item $(\bar w_{1,n}, \dots, \bar w_{\bar N,n}, \bar u_n)$ are solutions of \eqref{eqn model k} for $\beta_n \to +\infty$;
		\item $(w_{1,\beta}, \dots, w_{N,\beta}) \to (w_1, \dots, w_N)$ in $\C^{0,\alpha}\cap H^1(\overline \Omega)$, $u_\beta \to u$ in $\C^{2,\alpha}(\overline \Omega)$;
	\item $(w_1, \dots, w_N, u)$ solves \eqref{eqn segr model}  and
\[
	P(\bar w_1, \dots, \bar w_{\bar N}, \bar u) = \sup_{N \geq 1, (w_1, \dots, w_N, u) \in \mathscr{S}} P(w_1, \dots, w_N, u).
\]
	\end{itemize}
\end{itemize}
\end{theorem}
We stress the fact that we have imposed no conditions on $\beta > 0$ and $\mu > 0$. The proof of this theorem in contained in the previous results.

\section{The question of optimal repartition of densities}\label{sec emergence}
We now show that, in some settings, the configuration that maximizes the total population of predators (that is, the solution mentioned in Theorem \ref{thm maxim}) contains more than one non zero component of $w_i$. 

To this end, we first consider a simplified version of the system \eqref{eqn model red}, viz.
\begin{equation}\label{eqn model k simple first}
	\begin{cases}
		- \Delta w_1  = \left(- \omega + k u - \beta w_2\right) w_1 \\
		- \Delta w_2  = \left(- \omega + k u - \beta w_1\right) w_2 \\
		- \Delta u = \left(\lambda - k w_1 - k w_2 \right)u\\
		\partial_\nu w_i = \partial_\nu u = 0 &\text{ on $\partial \Omega$}.
	\end{cases}
\end{equation}
We observe that we are considering here a system with indistinguishable densities of predators (all the characterizing parameters in the equations are the same for all groups). Furthermore, we assume here that the parameter $\mu$ vanishes. An extension of the result in the case $\mu > 0$ will be presented later. We recall (see Proposition \ref{prop dock new} and Definition \ref{def: simple sol}) that \eqref{eqn model k simple first} only admits as simple solutions (that is, solutions with only one of $w_1$ and $w_2$ non zero)  the constant solution:
\[
	(W,U) = \left(\frac{\lambda}{k},\frac{\omega}{k}\right) \implies \int_{\Omega} W = \frac{\lambda}{k} |\Omega|, \int_{\Omega} U = \frac{\omega}{k} |\Omega|
\]
Under these assumptions, we have
\begin{lemma}\label{lem: first res opt 2 or more}
Let $(w_1, w_2, u)$ be an arbitrary solution of \eqref{eqn model k simple first} with $\beta > 0$. If all of its components are non zero and non constant, then 
\[
	\int_{\Omega} \sum_{i}^{2} w_i = \frac{\lambda}{k} |\Omega| + \frac{1}{k}\int |\nabla \log u|^2 >  \frac{\lambda}{k} |\Omega| 
\]
and
\[
	\int_{\Omega} u = \frac{\omega}{k} |\Omega| + \frac{\beta}{\lambda} \int w_1 w_2 +  \frac{\omega}{k\lambda}\int |\nabla \log u|^2 >  \frac{\omega}{k} |\Omega|.
\]
\end{lemma}
\begin{remark}
Equivalently, we could have compared the solutions of \eqref{eqn model k simple first} with $\beta > 0$ with any non zero solution in the case $\beta = 0$ (see Lemma \ref{lem trivial branch of sol}).
\end{remark}

\begin{proof}
We recall that, owing to Lemma \ref{lem u constant}, if one component is constant, so are the others. Thus, we can assume that $u$ is not constant. The existence of non constant solutions is already known thanks to Theorems \ref{thm bif} and \ref{thm bif an}. We consider the equation in $u$. By the maximum principle, $u > 0$ and thus we can divide the two sides of the equation by $u$ and integrate over $\Omega$. This yields
\begin{equation}\label{eqn integral id mu}
	\int_{\Omega} \sum_{i}^{2} w_i = \frac{1}{k} \int \left(\lambda + \frac{\Delta u}{u}\right) = \frac{\lambda}{k} |\Omega| + \frac{1}{k}\int |\nabla \log u|^2 >  \frac{\lambda}{k} |\Omega|.
\end{equation}
Here, the strict inequality follows by the fact that $u$ is not a constant. Similarly, integrating directly the equations of the system, summing them and using the previous identity, we get
\[
	 \int_{\Omega} u = \frac{\omega}{k} |\Omega|  + \frac{\beta}{\lambda} \int w_1 w_2 + \frac{\omega}{k\lambda}\int |\nabla \log u|^2 >  \frac{\omega}{k} |\Omega|
\]
and this concludes the proof.
\end{proof}

As a result, according to system \eqref{eqn model k simple first}, for any $\beta > 0$, competition is always advantageous both for the predators and for the preys. Indeed, the total population of predators (and preys) is greater in the case of two groups competing for the same territory, than in the case of only one group. 

We now extend this result to the model \eqref{eqn model red}, that is when $\mu > 0$. To this end, let us first observe that the same computations as above yield the identity
\[
	\begin{split}
	\int_{\Omega} \sum_{i}^{2} w_i &= \frac{\lambda}{k} |\Omega| - \frac{\mu}{k} \int u + \frac{1}{k}\int |\nabla \log u|^2 \\
	&= \frac{\lambda k - \mu \omega}{k^2} |\Omega| + \frac{\mu}{k}\left( \frac{\omega}{k}|\Omega| - \int u\right) + \frac{1}{k}\int |\nabla \log u|^2
	\end{split}
\]
which, by uniform convergence of the densities as $\beta \to +\infty$, is valid also in the limit case of segregation $\beta = +\infty$. Unlike the case $\mu = 0$, a direct comparison of the previous formula with the case of only one population of predators is not immediate. Indeed, in general, we can show that the second term in the last expression is negative, that is $\int u > \frac{\omega}{k}|\Omega|$.

As a result, we need to carefully estimate the various contributions in this identity, in order to show that, when $\mu$ is sufficiently small, the total population of predators is large in the case of two non trivial components $w_1$ and $w_2$. 

We wish to emphasize that this is a rather delicate task. Indeed, an a priori estimate on the solutions which is independent of $\mu$ may not be true, for two main reasons.
\begin{itemize}
	\item From the equation in $u$, we can only say a priori that $u \leq \lambda / \mu$. If $\mu \to 0$, we have no reason to conclude that the solutions of \eqref{eqn model red} converge to solutions of \eqref{eqn model k simple first}.
	\item One may wonder whether the previous bound is not sharp and that it may be achieved only by ``spurious''  solutions such as $(0,0,\lambda/\mu)$. But this assertion is not true in general. To see this we can recall that, by Theorem \ref{thm bif}, non constant (as it were ``genuine'') solutions bifurcate from 
\[
	\left(\frac{\lambda k - \mu \omega}{\mu \beta + 2 k^2}, \frac{\lambda k - \mu \omega}{\mu \beta + 2 k^2}, \frac{\lambda \beta + 2 k \omega}{\mu \beta + 2 k^2}\right) \qquad \text{ for $\beta\frac{\lambda k - \mu \omega}{\mu \beta + 2 k^2} = \gamma_n$}
\]
where $\gamma_n$ is the $n$-th eigenvalue of the Laplace operator with Neumann boundary conditions. For $\mu$ sufficiently small and $\gamma_n$ large (and, consequently, $\beta$ large), we have non constant solutions for which $u$ is close (at least) in the uniform topology to the upper bound $\lambda / \mu$.
\end{itemize}

We thus focus on the one-dimensional case, for which (see Theorem \ref{thm bif one}) we have already established the existence of segregated solutions and pointed out their symmetries (Remark \ref{rem bif sim}). As a result, for $\Omega = (-a,a)$, $a > 0$ and $\mu > 0$ sufficiently small, we have a continuum of solutions such that $w_1 - w_2$ vanishes only for $x = 0$. Sending the competition parameter to infinity $\beta \to +\infty$, by Lemma \ref{lem asymp positive 1d} we can start by considering the system 
\begin{equation}\label{eqn model 2 mu 1D}
	\begin{cases}
		- w''  = \left(- \omega + k u\right) w \\
		- u'' = \left(\lambda - \mu u - k w \right)u &\text{ in $(0,a)$}\\
		w(0) = w'(a) = u'(0) = u'(a) = 0
	\end{cases}
\end{equation}
for $\mu > 0$, for which the identity \eqref{eqn integral id mu} reduces to
\[
	\int_{0}^{a} w = \frac{\lambda}{k} a - \frac{\mu}{k} \int_0^a u + \frac{1}{k}\int_0^a |( \log u )'|^2
\]
\begin{proposition}\label{prop stima int}
Let $(w,u)$ be any classical solution of \eqref{eqn model 2 mu 1D} with both components non negative and nontrivial. For $\mu > 0$ sufficiently small
\[
	\int_{0}^{a} w > \frac{\lambda}{k} a.
\]
\end{proposition}

Let us observe that, since the solutions of \eqref{eqn model red} converge for $\beta \to +\infty$ to segregated solutions, the previous result implies an improvement of Theorem \ref{thm maxim}, and in particular we have
\begin{theorem}\label{2 is better}
Under the assumptions of Theorem \ref{thm maxim}, let us assume moreover that the coefficients in \eqref{eqn model k} do not depend on the index $i$ and that $\Omega = (a, b) \subset \R$. If $\mu > 0$ sufficiently small the solution of \eqref{eqn model k}  that maximizes 
\[
	P(w_1, w_2, \dots, w_N, u) = \int_{\Omega} \sum_{i=1}^N w_i
\]
has at least $N \geq 2$ non trivial components and $\beta > 0$.
\end{theorem}

We divide the proof of Theorem \ref{2 is better} in two separate results. Letting all the parameters in \eqref{eqn model 2 mu 1D} fixed a part from $\mu > 0$, we have
\begin{lemma}
Let $(w,u)$ be an arbitrary classical solution of \eqref{eqn model 2 mu 1D} with both components non negative and nion zero. For any $\eps > 0$ there exists $\bar \mu > 0$ such that
\[
	 \mu \int_0^a u \leq \eps \qquad \text{if $\mu \in (0,\bar\mu)$}.
\]
\end{lemma}
\begin{proof}
Let $(w_n, u_n)$ be a sequence of positive solutions to \eqref{eqn model 2 mu 1D} with $\mu = \mu_n \to 0$, and let us assume, by contradiction, that
\[
	\mu_n \int_0^a u_n > C  > 0.
\] 
The maximum principle, as already observed, implies that $u_n \leq \lambda / \mu_n$, hence
\[
	\mu_n \int_0^a u_n \leq a \lambda.
\]
Thus, we can assume that, up to a subsequence,
\[
	\lim_{n \to +\infty} \mu_n \int_0^a u_n = \gamma \in (0, a \lambda]
\]
for some constant $\gamma>0$. We can introduce the scaled functions $(\bar w_n, \bar u_n)$ as
\[
	\bar w_n := \left(\int_0^a (w_n')^2\right)^{-1/2} w_n, \quad \bar u_n := \left(\frac{1}{a} \int_0^a u_n \right)^{-1} u_n,
\]
which are solutions to
\begin{equation}\label{eqn model 2 mu 1D scaled}
	\begin{cases}
		- \bar w_n''  = \left(- \omega + k_n \bar u_n\right) \bar w_n \\
		- \bar u_n'' = \left(\lambda - \bar \mu_n \bar u_n - k_n' \bar w_n \right) \bar u_n &\text{ in $(0,a)$}\\
		 \bar w_n(0) =  \bar w_n'(a) = \bar u_n'(0) = \bar u_n'(a) = 0
	\end{cases}
\end{equation}
where we have defined
\[
	k_n := k \frac{1}{a} \int_0^a u_n,\quad \bar \mu_n := \mu_n \frac{1}{a} \int_0^a u_n, \quad k_n' := k \left(\int_0^a (w_n')^2\right)^{1/2}.
\]
We observe that, from the starting assumption, 
\[
	k_n \to +\infty \qquad \text{and} \quad  \bar \mu_n \to \bar \mu_\infty \in (0,\lambda),
\]
while we have no information on $k'_n$. Moreover, by definition and the Dirichlet boundary condition at zero, the sequence $\{\bar w_n\}_{n\in\N}$ is bounded in $H^1(0,a)$, and by positivity, also $\{\bar u_n\}_{n \in \N}$ is bounded in $H^1(0,a)$; Indeed, by testing the equation in $\bar u_n$ with $\bar u_n$ itself, we obtain
\[
	\int_{0}^a (\bar u_n')^2 + \mu_n \bar u_n^3 \leq \lambda \int_0^a \bar u_n^2
\]
and the claim follows from the assumption $\bar \mu_n \to C > 0$. By the embedding theorems we have that, up to a subsequence, both $\{\bar w_n\}_{n\in\N}$ and $\{\bar u_n\}_{n \in \N}$ converge uniformly in $(0,a)$ to their respective weak limits in the space $H^1(0,a)$, $\bar w_\infty$ and $\bar u_\infty$. Moreover, by renormalization and strong convergence, we have
\[
	\int_0^a \bar u_\infty = a
\]
and thus, in particular, $\bar u_\infty$ is non zero. Finally, from the equation in $\bar w_n$ we see that
\[
	k_n \int_0^a \bar u_n \bar w_n^2 = \int_0^a (\bar w_n')^2 + \omega \bar w_n^2 \leq C'.
\]
Since $k_n \to +\infty$, by the uniform convergence we have that
\[
	\bar u_n \bar w_n \to \bar u_\infty \bar w_\infty \equiv 0 \qquad \text{uniformly in $(0,a)$}.
\]

\noindent\textbf{Step 1)} We now proceed and exclude the possibility that the sequence $k'_n$ is bounded. By the uniform convergence we have
\[
	k'_n \bar u_n \bar w_n \to 0 \qquad \text{uniformly in $(0,a)$}.
\]
Passing to the limit in the equation satisfied by $\bar u_n$, it follows that $\bar u_\infty$ is a non trivial solution of
\begin{equation}\label{eqn u infty}
	\begin{cases}
		- \bar u_\infty'' = \left(\lambda - \bar  \mu_\infty \bar u_\infty \right) \bar u_\infty &\text{ in $(0,a)$}\\
		 \bar u_\infty'(0) = \bar u_\infty'(a) = 0.
	\end{cases}
\end{equation}
From Lemma \ref{lem max neumann}, $\bar u_\infty$ can thus be only the constant $\lambda / \bar \mu_\infty$, and finally, thanks to renormalization, $\bar \mu_\infty = \lambda$ and $\bar u_\infty \equiv 1$. By the uniform convergence of $\bar u_n$ to its limit, for $n$ sufficiently large we find
\[
	\bar u_n > \frac{1}{2} \qquad \text{in $(0,a)$}.
\]
Using this estimate in the equation for $\bar w_n$, we get that
\[
	\begin{cases}
		- \bar w_n''  = \left(- \omega + k_n \bar u_n\right) \bar w_n > \left(- \omega + k_n/2 \right) \bar w_n  \\
		 \bar w_n(0) =  \bar w_n'(a) = 0.
	\end{cases}	
\]
Since $k_n \to +\infty$, for $n$ large enough we find
\[
	\left(- \omega + k_n/2 \right) > \left(\frac{\pi}{2a}\right)^2.
\]
The right hand side is the principal eigenvalue of the operator. By the comparison principle we see that $\bar w_n \equiv 0$, against the assumption that the functions $(w_n,u_n)$ are positive in $(0,a)$. This possibility is thus ruled out.

\noindent\textbf{Step 2)} Hence it must be the case that $k'_n \to +\infty$: similarly we can show that $k_n' \geq C k_n$. Indeed, let us assume by contradiction that $k_n' / k_n \to 0$. Testing the equation in $\bar w_n$ with $\varphi$ smooth and compactly supported, we find
\[
	k_n \int_0^a \bar u_n \bar w_n \varphi = \int_a^a \left( \bar w_n' \varphi + \omega \bar w_n \varphi \right)
\]
and the right-hand side is bounded in $n$, so that
\[
	k'_n \int_0^a \bar u_n \bar w_n \varphi =  \frac{k'_n}{k_n} k_n \int_0^a \bar u_n \bar w_n \varphi  \to 0
\]
for all test function $\varphi$. It follows that the weak and uniform limit  $\bar u_\infty$ of $\bar u_n$ solves again \eqref{eqn u infty}, and thus
\[
	\bar u_n \to 1 \qquad \text{uniformly in $(0,a)$}.
\]
We can then reach a contradiction as in Step 1).

\noindent\textbf{Step 3)} We now show that $k_n \geq C k'_n$. Again by contradiction, let us assume that $k_n / k'_n \to 0$. Integrating the equation in $\bar u_n$ we obtain
\[
	k'_n\int_0^a \bar u_n \bar w_n = \int_0^a \bar u_n \left( \lambda -\mu_n \bar u_n \right)
\]
and the right-hand side is bounded uniformly in $n$. It follows that
\[
	k_n\int_0^a \bar u_n \bar w_n = \frac{k_n}{k_n'} k_n' \int_0^a \bar u_n \bar w_n  \to 0.
\]
But then, testing the equation in $\bar w_n$ with $\bar w_n$ itself, we obtain, thanks to the renormalization of $\bar w_n$ and the uniform convergence,
\[
	 0 < C' < \int_0^a (\bar w_n')^2 + \omega \bar w_n^2 = k_n \int_0^a \bar u_n \bar w_n^2 \leq \|\bar w_n \|_{L^\infty} \cdot k_n \int_0^a \bar u_n \bar w_n \to 0
\]
a contradiction.

\noindent\textbf{Step 4)} In summary, so far we have shown that
\[
	k'_n \sim k_n \to +\infty.
\]
We already know that, up to a subsequence, the sequence  $\{\bar w_n\}_{n \in \N}$ converges uniformly to continuous function $\bar w_\infty$. We now show that $\bar w_\infty \equiv 0$. We argue by contradiction and assume that this is not the case. Then there exist $0 \leq x_0 < x_1 \leq a$ such that
\[
	\inf_{x \in [x_0,x_1]} \bar w_n > C > 0 \qquad \text{for all $n$ sufficiently large}
\]
Then for some positive constants $A, B$ and $n$ sufficiently large, we have
\[
	\begin{cases}
		 \bar u_n < A &\text{ in $(0,a)$} \\
		- \bar u_n'' = \left(\lambda - \mu_n \bar u_n - k_n' \bar w_n \right) \bar u_n < - B k_n' \bar u_n &\text{ in $(x_0,x_1)$}.
	\end{cases}
\]
By comparison with the super-solution
\[
	x \mapsto \left.A \cosh\left[ (B k'_n)^{1/2} \left(x - \frac{x_0+x_1}{2}\right)\right] \right/  \cosh\left[ (B k'_n)^{1/2} \left(\frac{x_1-x_0}{2}\right)\right]
\]
we have that, for $\eps > 0$ small
\[
	\sup_{x \in [x_0+\eps,x_1-\eps]} \bar u_n \leq A \frac{ \cosh\left[ (B k'_n)^{1/2} \left( \frac{x_1-x_0}{2} -\eps\right)\right] }{ \cosh\left[ (B k'_n)^{1/2} \left(\frac{x_1-x_0}{2}\right)\right]}.
\]
Turning back to the equation in $\bar w_n$, and recalling that $k_n = O(k'_n)$, we can pass to the limit and obtain
\[
	- \bar w_\infty''  = - \omega \bar w_\infty  \qquad \text{ in $[x_0  + \eps ,x_1 - \eps]$}.
\]
We now observe that the previous reasoning holds is true for any $\eps$ and any interval of positivity $[x_0,x_1]$ of $\bar w_\infty$. As a result, in any interval of positivity $[x_0,x_1] \subset [0,a]$ we have
\[
	\text{either} \quad \begin{cases}
		- \bar w_\infty''  = - \omega \bar w_\infty \\
		 \bar w_\infty(x_0) =  \bar w_\infty(x_1) = 0
	\end{cases}\qquad \text{or} \quad
	\begin{cases}
		- \bar w_\infty''  = - \omega \bar w_\infty \\
		 \bar w_\infty(x_0) =  \bar w_\infty'(x_1) = 0
	\end{cases} 
\]
In both case, arguing as in Step 1), we see that $\bar w_\infty \equiv 0$ in $[x_0,x_1]$, meaning that there are no intervals in $(0,a)$ where $\bar w_\infty$ is positive, and thus
\[
	\bar w_n \to 0 \qquad \text{uniformly in $(0,a)$}.
\]
Now we can repeat the reasoning of Step 3): integrating the equation in $\bar u_n$ we obtain
\[
	k_n\int_0^a \bar u_n \bar w_n \sim k'_n\int_0^a \bar u_n \bar w_n = \int_0^a \bar u_n \left( \lambda -\mu_n \bar u_n \right).
\]
From the equation in $\bar w_n$ and the uniform limit proved before, we have
\[
	 0 < C' < \int_0^a (\bar w_n')^2 + \omega \bar w_n^2 = k_n \int_0^a \bar u_n \bar w_n^2 \leq \|\bar w_n \|_{L^\infty} \cdot k_n \int_0^a \bar u_n \bar w_n \to 0,
\]
which is a contradiction.
\end{proof}

\begin{lemma}
Let $(w,u)$ be an arbitrary classical solution of \eqref{eqn model 2 mu 1D} with both components non negative and non zero. There exist two constants $C > 0$ and $\bar \mu > 0$ such that
\[
	 \int_0^a |(\log u)'|^2 > C  \qquad \text{if $\mu \in (0,\bar\mu)$}.
\]
\end{lemma}

\begin{proof}
Let us consider a sequence $(w_n, u_n)$ of positive solutions to \eqref{eqn model 2 mu 1D} for $\mu = \mu_n$ such that
\[
	\lim_{n\to+\infty}\int_0^a |(\log u_n)'|^2 = 0
\]
By the embedding theorems we know that
\[
	\lim_{n\to+\infty}\frac{\sup_{x,y \in (0,a)}|u_n(x)-u_n(y)|}{ \|u_n\|_{L^\infty(0,a)} } = 0.
\]
Moreover, since $u_n$ is positive, we have
\begin{equation}\label{eqn rel osc to 0}
	\lim_{n\to+\infty} \frac{\inf_{(0,a)} u_n}{\sup_{(0,a)} u_n} = 1.
\end{equation}
Let us now show that \eqref{eqn rel osc to 0} implies that $u_n$ converges to a positive and finite constant. For this it suffices to show that $\{u_n\}_{n\in\N}$ is uniformly bounded from above and away from zero. Indeed, if $\inf u_n \to +\infty$, then for $n$ sufficiently large
\[
	\begin{cases}
		- w_n''  = \left(- \omega + k u_n\right) w_n > \left( \frac{\pi}{2a}\right)^2 w_n \\
		w_n(0) = w_n'(a) = 0
	\end{cases} \implies w_n \equiv 0.
\]
On the other hand, if $\sup u_n \to 0$, then for $n$ sufficiently large
\[
	\begin{cases}
		- w_n''  = \left(- \omega + k u_n\right) w_n < \left( \frac{\pi}{2a}\right)^2 w_n \\
		w_n(0) = w_n'(a) = 0
	\end{cases} \implies w_n \equiv 0
\]
and in both cases we reach a contradiction with the positivity of $w_n$. This, together with \eqref{eqn rel osc to 0} implies that $u_n$ converges uniformly to a positive constant $C$. Again by the equation in $w_n$, we see that necessarily
\[
	u_n \to \frac{\omega}{k} + \frac{1}{k}\left( \frac{\pi}{2a}\right)^2 \qquad \text{uniformly in $(0,a)$}.
\]
Up to a renormalization, we also infer that
\[
	\left(\|w_n\|_{L^\infty(0,a)}\right)^{-1} w_n \to \sin \left(\frac{\pi}{2a} x\right)
\]
strongly in $H^1(0,a)$ and also uniformly. Integrating the equation in $u_n$, we have
\[
	k \int_{0}^a u_n w_n = \lambda \int_0^a u_n - \mu_n \int_0^a u_n^2.
\]
Thus, we have that $\{w_n\}_{n\in\N}$ is uniformly bounded in $(0,a)$, so that
\[
	w_n \to C \sin \left(\frac{\pi}{2a} x\right)
\]
for a non negative constant $C$. Consequently, using these information in the equation satisfied by $u_n$, we have that $u_n$ is bounded uniformly in $H^1(0,a)$, and we can thus take the weak limit of the equation to derive 
\[
	0 = \left[\lambda - k C \sin \left(\frac{\pi}{2a} x\right) \right] \left[\frac{\omega}{k} + \frac{1}{k}\left( \frac{\pi}{2a}\right)^2  \right],
\]
which is impossible to enforce and gives us the desired contradiction. The proof of Theorem \ref{2 is better} is thereby complete.
\end{proof}

\bigskip
\noindent \textbf{Acknowledgements:}  This work has been supported by the ERC Advanced Grant 2013 n. 321186 ``ReaDi -- Reaction-Diffusion Equations, Propagation and Modelling'' held by Henri Berestycki while Alessandro Zilio was post-doc fellow of the research grant. This work was also partially supported by the French National Research Agency (ANR), within  project NONLOCAL ANR-14-CE25-0013.

%\bibliographystyle{plain}
%\bibliography{biblio}

\begin{thebibliography}{10}

\bibitem{BerestyckiZilio_NN}
Henri Berestycki and Alessandro Zilio.
\newblock Predators-prey models with competition {P}art {II}: uniform
  regularity estimates.
\newblock {\em In preparation}.

\bibitem{BerestyckiZilio_RR}
Henri Berestycki and Alessandro Zilio.
\newblock Predators-prey models with competition {P}art {III}: classification
  of solutions and rigidity of the model.
\newblock {\em In preparation}.

\bibitem{BerestyckiZilio_FF}
Henri Berestycki and Alessandro Zilio.
\newblock Predators-prey models with competition {P}art {IV}: the properties of
  the parabolic system.
\newblock {\em In preparation}.

\bibitem{BZ_ecology}
Henri Berestycki and Alessandro Zilio.
\newblock Predators-prey models with competition: the emergence of packs and
  territoriality.
\newblock {\em Preprint}.

\bibitem{BuffoniToland}
Boris Buffoni and John Toland.
\newblock {\em Analytic theory of global bifurcation}.
\newblock Princeton Series in Applied Mathematics. Princeton University Press,
  Princeton, NJ, 2003.
\newblock An introduction.

\bibitem{CaffKarLin}
L.~A. Caffarelli, A.~L. Karakhanyan, and Fang-Hua Lin.
\newblock The geometry of solutions to a segregation problem for nondivergence
  systems.
\newblock {\em J. Fixed Point Theory Appl.}, 5(2):319--351, 2009.

\bibitem{ContiTerraciniVerzini_AdvMat_2005}
Monica Conti, Susanna Terracini, and G.~Verzini.
\newblock Asymptotic estimates for the spatial segregation of competitive
  systems.
\newblock {\em Adv. Math.}, 195(2):524--560, 2005.

\bibitem{CTV_indiana}
Monica Conti, Susanna Terracini, and Gianmaria Verzini.
\newblock A variational problem for the spatial segregation of
  reaction-diffusion systems.
\newblock {\em Indiana Univ. Math. J.}, 54(3):779--815, 2005.

\bibitem{ConwayHoffSmoller}
Edward Conway, David Hoff, and Joel Smoller.
\newblock Large time behavior of solutions of systems of nonlinear
  reaction-diffusion equations.
\newblock {\em SIAM J. Appl. Math.}, 35(1):1--16, 1978.

\bibitem{ConwaySmoller}
Edward~D. Conway and Joel~A. Smoller.
\newblock Diffusion and the predator-prey interaction.
\newblock {\em SIAM J. Appl. Math.}, 33(4):673--686, 1977.

\bibitem{Dancer_71}
E.~N. Dancer.
\newblock Bifurcation theory in real {B}anach space.
\newblock {\em Proc. London Math. Soc. (3)}, 23:699--734, 1971.

\bibitem{Dancer_73}
E.~N. Dancer.
\newblock Bifurcation theory for analytic operators.
\newblock {\em Proc. London Math. Soc. (3)}, 26:359--384, 1973.

\bibitem{DancerDu}
E.~N. Dancer and Yi~Hong Du.
\newblock Competing species equations with diffusion, large interactions, and
  jumping nonlinearities.
\newblock {\em J. Differential Equations}, 114(2):434--475, 1994.

\bibitem{DaWaZa_Dynamics}
E.~N. Dancer, Kelei Wang, and Zhitao Zhang.
\newblock Dynamics of strongly competing systems with many species.
\newblock {\em Trans. Amer. Math. Soc.}, 364(2):961--1005, 2012.

\bibitem{Dockery}
Jack Dockery, Vivian Hutson, Konstantin Mischaikow, and Mark Pernarowski.
\newblock The evolution of slow dispersal rates: a reaction diffusion model.
\newblock {\em J. Math. Biol.}, 37(1):61--83, 1998.

\bibitem{HelfHofOstTerr_2009}
B.~Helffer, T.~Hoffmann-Ostenhof, and S.~Terracini.
\newblock Nodal domains and spectral minimal partitions.
\newblock {\em Ann. Inst. H. Poincar{\'e} Anal. Non Lin{\'e}aire},
  26(1):101--138, 2009.

\bibitem{Mimura}
Masayasu Mimura.
\newblock Asymptotic behaviors of a parabolic system related to a planktonic
  prey and predator model.
\newblock {\em SIAM J. Appl. Math.}, 37(3):499--512, 1979.

\bibitem{Rabinowitz_JFA}
Paul~H. Rabinowitz.
\newblock Some global results for nonlinear eigenvalue problems.
\newblock {\em J. Functional Analysis}, 7:487--513, 1971.

\bibitem{RabBere}
P.H. Rabinowitz and H.~Berestycki.
\newblock {\em Th{\'e}orie du degr{\'e} topologique et applications {\`a} des
  probl{\'e}mes aux limites non lin{\'e}aires}.
\newblock Universit{\'e} Paris VI, Laboratoire analyse num{\'e}rique, 1975.

\bibitem{Ruelle}
D.~{Ruelle}.
\newblock {Bifurcations in the presence of a symmetry group}.
\newblock {\em Archive for Rational Mechanics and Analysis}, 51:136--152,
  January 1973.

\bibitem{Sattinger}
D.~Sattinger.
\newblock {\em Branching in the Presence of Symmetry}.
\newblock Society for Industrial and Applied Mathematics, 1983.

\bibitem{SoaveZilio_ARMA}
Nicola Soave and Alessandro Zilio.
\newblock Uniform bounds for strongly competing systems: the optimal
  {L}ipschitz case.
\newblock {\em Arch. Ration. Mech. Anal.}, 218(2):647--697, 2015.

\bibitem{TaTe}
Hugo Tavares and Susanna Terracini.
\newblock Regularity of the nodal set of segregated critical configurations
  under a weak reflection law.
\newblock {\em Calc. Var. Partial Differential Equations}, 45(3-4):273--317,
  2012.

\bibitem{spirals}
Susanna Terracini, Gianmaria Verzini, and Alessandro Zilio.
\newblock Spiraling asymptotic profiles of competition-diffusion systems.
\newblock {\em Preprint}.

\bibitem{Volterra}
Vito Volterra.
\newblock Variations and fluctuations of the number of individuals in animal
  species living together.
\newblock {\em Journal du Cons. Int. Explor. Mer}, 3:5--51, 1928.

\end{thebibliography}

\end{document}